\documentclass[12pt]{amsart}
\usepackage{amsmath}
\usepackage{amsfonts}
\usepackage{amssymb}
\usepackage{mathtools}
\usepackage{bbding}
\usepackage{stmaryrd}
\usepackage{txfonts}
\usepackage{graphicx}
\usepackage{epsfig}
\usepackage{xypic}
\usepackage{tikz}
\usepackage{longtable}
\usepackage[all]{xy}
\usepackage{pgflibraryarrows}
\usepackage{pgflibrarysnakes}
\usepackage[shortlabels]{enumitem}
\usepackage{ifpdf}
\ifpdf
\usepackage[colorlinks,final,backref=page,hyperindex]{hyperref}
\else
\usepackage[colorlinks,final,backref=page,hyperindex,hypertex]{hyperref}
\fi
\usepackage{graphicx}
\usepackage{epstopdf}
\usepackage{epsfig}
\usepackage{pdfsync}    

\usepackage{xspace}

\newtheorem{theorem}{Theorem}[section]
\newtheorem{lemma}[theorem]{Lemma}
\newtheorem{corollary}[theorem]{Corollary}

\newtheorem{proposition}[theorem]{Proposition}

\theoremstyle{definition}
\newtheorem{definition}[theorem]{Definition}
\newtheorem{remark}[theorem]{Remark}
\newtheorem{example}[theorem]{Example}

\newcommand{\nc}{\newcommand}
\newcommand{\delete}[1]{}

\delete{

}

\def\bc{\begin{center}}
	\def\ec{\end{center}}

\nc{\tred}[1]{\textcolor{red}{#1}}
\nc{\tblue}[1]{\textcolor{blue}{#1}} \nc{\tgreen}[1]{\textcolor{green}{#1}} \nc{\tpurple}[1]{\textcolor{purple}{#1}} \nc{\btred}[1]{\textcolor{red}{\bf #1}} \nc{\btblue}[1]{\textcolor{blue}{\bf #1}} \nc{\btgreen}[1]{\textcolor{green}{\bf #1}} \nc{\btpurple}[1]{\textcolor{purple}{\bf #1}}

\newcommand{\efootnote}[1]{}

\nc{\mlabel}[1]{\label{#1}}  
\nc{\mcite}[1]{\cite{#1}}  
\nc{\mref}[1]{\ref{#1}}  
\nc{\meqref}[1]{\eqref{#1}}  
\nc{\mbibitem}[1]{\bibitem{#1}} 

\delete{
	\nc{\mlabel}[1]{\label{#1}  
		{\hfill \hspace{1cm}{\bf{{\ }\hfill(#1)}}}}
	\nc{\mcite}[1]{\cite{#1}{{\bf{{\ }(#1)}}}}  
	\nc{\mref}[1]{\ref{#1}{{\bf{{\ }(#1)}}}}  
	\nc{\meqref}[1]{\eqref{#1}{{\bf{{\ }(#1)}}}}  
	\nc{\mbibitem}[1]{\bibitem[\bf #1]{#1}} 
}

\renewcommand\geq{\geqslant}
\renewcommand\leq{\leqslant}

\renewcommand\bar[1]{\overline{#1}}

\nc{\name}[1]{{\bf #1}}

\nc{\tforall}{\quad \text{ for all }}

\nc{\mre}{\text{Re}\,}

\nc{\mim}{\text{im}\,}
\nc{\nz}{\varepsilon}
\nc{\Id}{\mathrm{Id}}

\nc{\DO}{\text{DO}}
\nc{\IDO}{\text{IDO}}

\nc{\IEO}{\mathrm{IEO}}

\nc{\mnoindent}{\smallskip\noindent}

\nc{\lnvkv}{\triangleleft}

\nc{\bin}[2]{ (_{\stackrel{\scs{#1}}{\scs{#2}}})}  
\nc{\binc}[2]{ \left (\!\! \begin{array}{c} \scs{#1}\\
		\scs{#2} \end{array}\!\! \right )}  
\nc{\bincc}[2]{  \left ( {\scs{#1} \atop
		\vspace{-1cm}\scs{#2}} \right )}  
\nc{\bs}{\bar{S}} \nc{\cosum}{\sqsubset} \nc{\la}{\longrightarrow} \nc{\rar}{\rightarrow} \nc{\dar}{\downarrow} \nc{\dprod}{**} \nc{\dap}[1]{\downarrow \rlap{$\scriptstyle{#1}$}} \nc{\md}[1]{\bar{#1}} \nc{\uap}[1]{\uparrow \rlap{$\scriptstyle{#1}$}} \nc{\defeq}{\stackrel{\rm def}{=}} \nc{\disp}[1]{\displaystyle{#1}} \nc{\dotcup}{\ \displaystyle{\bigcup^\bullet}\ } \nc{\gzeta}{\bar{\zeta}} \nc{\hcm}{\ \hat{,}\ } \nc{\hts}{\hat{\otimes}} \nc{\barot}{{\otimes}} \nc{\free}[1]{\bar{#1}} \nc{\uni}[1]{\tilde{#1}} \nc{\hcirc}{\hat{\circ}} \nc{\leng}{\ell} \nc{\lleft}{[} \nc{\lright}{]} \nc{\lc}{\lfloor} \nc{\rc}{\rfloor}
\nc{\curlyl}{\left \{ \begin{array}{c} {} \\ {} \end{array}
	\right.  \!\!\!\!\!\!\!}
\nc{\curlyr}{ \!\!\!\!\!\!\!
	\left. \begin{array}{c} {} \\ {} \end{array}
	\right \} }
\nc{\longmid}{\left | \begin{array}{c} {} \\ {} \end{array}
	\right. \!\!\!\!\!\!\!}
\nc{\onetree}{\bullet} \nc{\ora}[1]{\stackrel{#1}{\rar}}
\nc{\ola}[1]{\stackrel{#1}{\la}}
\nc{\ot}{\otimes} \nc{\mot}{{{\boxtimes\,}}} \nc{\otm}{\overline{\boxtimes}} \nc{\sprod}{\bullet} \nc{\scs}[1]{\scriptstyle{#1}} \nc{\mrm}[1]{{\rm #1}} \nc{\msum}{\sum\limits}
\nc{\margin}[1]{\marginpar{\rm #1}}   
\nc{\dirlim}{\displaystyle{\lim_{\longrightarrow}}\,} \nc{\invlim}{\displaystyle{\lim_{\longleftarrow}}\,} \nc{\mvp}{\vspace{0.3cm}} \nc{\tk}{^{(k)}} \nc{\tp}{^\prime} \nc{\ttp}{^{\prime\prime}} \nc{\svp}{\vspace{2cm}} \nc{\vp}{\vspace{8cm}} \nc{\proofbegin}{\noindent{\bf Proof: }}
\nc{\proofend}{$\blacksquare$ \vspace{0.3cm}}
\nc{\modg}[1]{\!<\!\!{#1}\!\!>}
\nc{\intg}[1]{F_C(#1)} \nc{\lmodg}{\!<\!\!} \nc{\rmodg}{\!\!>\!} \nc{\cpi}{\widehat{\Pi}}
\nc{\sha}{{\mbox{\cyr X}}}  
\nc{\shap}{{\mbox{\cyrs X}}} 
\nc{\shpr}{\diamond}    
\nc{\shp}{\ast} \nc{\shplus}{\shpr^+}
\nc{\shprc}{\shpr_c}    
\nc{\msh}{\ast} \nc{\zprod}{m_0} \nc{\oprod}{m_1} \nc{\vep}{\varepsilon} \nc{\labs}{\mid\!} \nc{\rabs}{\!\mid}
\nc{\astarrow}{\overset{\raisebox{-3pt}{$\ast$}}{\rightarrow}}

\nc{\sqsym}{Stirling quasisymmetric function\xspace}
\nc{\sqsyms}{Stirling quasisymmetric functions\xspace}

\nc{\EEsym}{\mathbb{E}sym}
\nc{\Sym}{\mrm{Sym}}
\nc{\NSym}{\mrm{NSym}}
\nc{\QSym}{\mrm{QSym}}
\nc{\RQSym}{\mrm{RQSym}}
\nc{\RenQSym}{\mrm{WCQSym}}	
\nc{\DQSym}{\mrm{DQSym}}
\nc{\WDQSym}{\mrm{WDQSym}}
\nc{\DLQSym}{\mrm{DLQSym}}
\nc{\ZQSym}{\mrm{ZQSym}}
\nc{\Ensym}{\mrm{ENSym}}
\nc{\Wcsym}{\mrm{WCSym}}
\nc{\LWQSym}{\mrm{LWQSym}}
\nc{\LWCQSym}{\mrm{\mathrm{LWQSym}}}
\nc{\Wcqsym}{\mrm{QSym}_{\widetilde{\mathbb{N}}}}
\nc{\Syms}{symmetric functions\xspace}
\nc{\eqsym}{extended quasisymmetric function\xspace}
\nc{\eqsyms}{extended quasisymmetric functions\xspace}
\nc{\Eqsyms}{Extended Quasisymmetric functions\xspace}
\nc{\Esyms}{Extended symmetric functions\xspace}
\nc{\sgqsym}{quasisymmetric function with semigroup exponents\xspace}
\nc{\sgqsyms}{quasisymmetric functions with semigroup exponents\xspace}
\nc{\Sgqsyms}{Quasisymmetric functions with semigroup exponents\xspace}
\nc{\SGQSYM}{\mrm{SGQSYM}}
\nc{\emzv}{extended multiple zeta value}
\nc{\emzvs}{extended multiple zeta values}
\nc{\sgfps}{formal power series with semigroup exponent\xspace}
\nc{\NSymg}{\mathrm{NSym}_\gp}
\nc{\zqsym}{zeta-quasisymmetric }
\nc{\gslwqsym}{Stirling left weak quasisymmetric function\xspace}
\nc{\gslwqsyms}{Stirling left weak quasisymmetric functions\xspace}
\nc{\ulwb}{upper-left weak bicomposition\xspace}
\nc{\ulwbs}{upper-left weak bicompositions\xspace}

\nc{\parr}{\rm Par}
\nc{\wpar}{\rm WPar}
\nc{\wcomp}{\large{\VDash}}
\nc{\Ker}{\ker}

\nc{\dth}{d} \nc{\mmbox}[1]{\mbox{\ #1\ }} \nc{\fp}{\mrm{FP}} \nc{\rchar}{\mrm{char}} \nc{\Fil}{\mrm{Fil}} \nc{\Mor}{Mor\xspace} \nc{\gmzvs}{gMZV\xspace} \nc{\gmzv}{gMZV\xspace} \nc{\mzv}{MZV\xspace} \nc{\mzvs}{MZVs\xspace}
\nc{\MZV}{\mathrm{MZV}}
\nc{\Hom}{\mrm{Hom}} \nc{\id}{\mrm{id}} \nc{\im}{\mrm{im}} \nc{\incl}{\mrm{incl}}  \nc{\mchar}{\rm char}

\nc{\Alg}{\mathbf{Alg}} \nc{\Bax}{\mathbf{Bax}} \nc{\bff}{\mathbf f} \nc{\bfk}{{\bf k}} \nc{\bfone}{{\bf 1}} \nc{\bfx}{\mathbf x} \nc{\bfy}{\mathbf y}
\nc{\base}[1]{\bfone^{\otimes ({#1}+1)}} 
\nc{\Cat}{\mathbf{Cat}} \delete{}
\nc{\detail}{\marginpar{\bf More detail}
	\noindent{\bf Need more detail!}
	\svp}
\nc{\Int}{\mathbf{Int}} \nc{\Mon}{\mathbf{Mon}}
\nc{\rbtm}{{shuffle }} \nc{\rbto}{{Rota-Baxter }} \nc{\remarks}{\noindent{\bf Remarks: }} \nc{\Rings}{\mathbf{Rings}} \nc{\Sets}{\mathbf{Sets}}
\nc{\balpha}{\mathbf{\alpha}}

\nc{\BA}{{\mathbb A}} \nc{\CC}{{\mathbb C}} \nc{\DD}{{\mathbb D}} \nc{\EE}{{\mathbb E}} \nc{\FF}{{\mathbb F}} \nc{\GG}{{\mathbb G}} \nc{\HH}{{\mathbb H}} \nc{\LL}{{\mathbb L}} \nc{\NN}{{\mathbb N}} \nc{\KK}{{\mathbb K}} \nc{\PP}{{\mathbb P}} \nc{\QQ}{{\mathbb Q}} \nc{\RR}{{\mathbb R}} \nc{\TT}{{\mathbb T}} \nc{\VV}{{\mathbb V}} \nc{\ZZ}{{\mathbb Z}}

\nc{\cala}{{\mathcal A}} \nc{\calc}{{\mathcal C}} \nc{\cald}{{\mathcal D}} \nc{\cale}{{\mathcal E}} \nc{\calf}{{\mathcal F}} \nc{\calg}{{\mathcal G}} \nc{\calh}{{\mathcal H}} \nc{\cali}{{\mathcal I}} \nc{\call}{{\mathcal L}} \nc{\calm}{{\mathcal M}} \nc{\caln}{{\mathcal N}} \nc{\calo}{{\mathcal O}} \nc{\calp}{{\mathcal P}} \nc{\calr}{{\mathcal R}} \nc{\cals}{{\mathcal S}} \nc{\calt}{{\mathcal T}} \nc{\calw}{{\mathcal W}} \nc{\calk}{{\mathcal K}} \nc{\calx}{{\mathcal X}}
\nc{\calz}{{\mathcal Z}}

\nc{\fraka}{{\mathfrak a}} \nc{\frakA}{{\mathfrak A}} \nc{\frakb}{{\mathfrak b}} \nc{\frakB}{{\mathfrak B}}
\nc{\frakc}{{\mathfrak c}}  \nc{\frakD}{{\mathfrak D}}
\nc{\frakH}{{\mathfrak H}}
\nc{\frakh}{{\mathfrak h}} \nc{\frakM}{{\mathfrak M}}
\nc{\frakO}{{\mathfrak O}}
\nc{\frakE}{{\mathfrak E}}
\nc{\bfrakM}{\overline{\frakM}} \nc{\frakm}{{\mathfrak m}} \nc{\frakP}{{\mathfrak P}} \nc{\frakN}{{\mathfrak N}} \nc{\frakp}{{\mathfrak p}} \nc{\frakS}{{\mathfrak S}}
\nc{\frakk}{{\mathfrak k}}
\nc{\frakx}{{\mathfrak x}}
\nc{\frakl}{{\mathfrak l}} \nc{\ox}{\bar{\frakx}} \nc{\frakX}{{\mathfrak X}} \nc{\fraky}{{\mathfrak y}} \nc\dop{\delta}
\nc{\Reduce}{{\rm Red}}

\font\cyr=wncyr10 \font\cyrs=wncyr7
\nc{\redt}[1]{\textcolor{red}{#1}}
\nc{\li}[1]{\textcolor{red}{#1}}
\nc{\lir}[1]{\textcolor{red}{Li:#1}}
\nc{\ap}[1]{\textcolor{blue}{#1}}
\nc{\apr}[1]{\textcolor{blue}{AP:#1}}

\nc{\RBO}{\mathrm{IO}}
\nc{\mrep}{{\mathrm{Id}}}

\nc{\SD}{{\text{SD}}}

\nc{\Fix}{{\text{Fix}}}

\nc{\rb}{integral\xspace}
\nc{\Rb}{Integral\xspace}
\nc{\mrb}{Rota-Baxter\xspace}

\nc{\wvec}[2]{{\scriptsize{\Big [ \!\!\begin{array}{c} #1 \\ #2 \end{array} \!\! \Big ]}}}

\nc{\bwvec}[2]{\Big(\wvec{#1}{#2}\Big)}

\nc{\jwvec}[2]{{\scriptsize{\Big [ \!\!\begin{array}{cccccccccccccc} #1 \\ #2 \end{array} \!\! \Big ]}}}
\nc{\bjwvec}[2]{\Big(\jwvec{#1}{#2}\Big)}

\begin{document}
	
\title[Integral operators on lattices]{Integral operators on lattices}

\author{Aiping Gan}
\address{School of Mathematics and Statistics,
Jiangxi Normal University, Nanchang, Jiangxi 330022, P.~R. China}
\email{ganaiping78@163.com}

\author{Li Guo}
\address{Department of Mathematics and Computer Science, Rutgers University, Newark, NJ 07102, USA}
\email{liguo@rutgers.edu}

\author{Shoufeng Wang}
\address{Department of Mathematics, Yunnan Normal University, Yunnan, Kunming, 650500, P.~R. China}
\email{wsf1004@163.com}
	
\hyphenpenalty=8000
	
\date{\today}
	
\begin{abstract} 
As an abstraction and generalization of the integral operator in analysis, integral operators (known as Rota-Baxter operators of weight zero) on associative algebras and Lie algebras have played an important role in mathematics and physics. This paper initiates the study of integral operators on lattices and the resulting \mrb lattices (of weight zero). We show that properties of lattices can be characterized in terms of their integral operators. We also display a large number of integral operators on any given lattice and classify the isomorphism classes of integral operators on some common classes of lattices. We further investigate structures on semirings derived from differential and integral operators on lattices.
\end{abstract}
	
\subjclass[2010]{
	06B75,	
	06C05,	
	06D75,	
	06B20, 	
	17B38,  
	13N15,	
	16Y60	
}
	
\keywords{Lattice, integral operator, derivation, \mrb lattice, differential lattice, semiring, Novikov semiring, dendriform semiring}
	
\maketitle

\vspace{-.7cm}
	
\tableofcontents
	
\hyphenpenalty=8000 \setcounter{section}{0}
	
	
\allowdisplaybreaks

\section {Introduction}	
\mlabel{sec:intr}

This paper introduces the notion of \rb operators on lattices and studies their role in understanding lattices, their classification and their derived structures. 

As is well known, 
 the derivation, or differential operator, and integral operator are fundamental in analysis and its broad applications. As an abstraction of the derivation, the notion of a differential algebra was introduced in the 1930's by Ritt~\mcite{Ri}, to be a field $A$ carrying a linear operator $d$ satisfying an abstraction of the Leibniz rule for the derivation:
$$d(uv)=d(u) v+ud(v) \quad \text{ for all } u, v\in A.$$
Thus $d$ is still called a differential operator. 
The theory of differential algebra for fields and more generally for commutative algebras has since been developed into a mature area of mathematical research including differential Galois theory, differential algebraic geometry and differential algebraic groups~\mcite{CGKS,Ko,SP}. Furthermore, differential algebra has found profound applications in arithmetic geometry, logic and computational algebra. 
	
The notion of derivations on lattices was first introduced by Szasz \mcite{sz}. 
There, a derivation on a lattice $(L,\vee,\wedge)$ is a map $d:L\to L$ satisfying
\begin{equation}
	d(x\vee y)=d(x)\vee d(y), \quad d(x\wedge y)=(d(x)\wedge y)\vee (x\wedge d(y)) \tforall x, y\in L.
\mlabel{eq:00}
\end{equation}
More recently, a less restricted notion of derivations was studied with motivation from information science~\mcite{xin2,xin1}, without requiring the first condition. 
This study was continued in~\mcite{GG}, where the notion of a differential lattice was formally introduced and then studied from the viewpoint of universal algebra. 

Originated from a probability study of G.~Baxter~\mcite{Bax} and promoted by G.-C.~Rota in its early stage, a Rota-Baxter algebra is an associative algebra together with a linear operator satisfying a variation of the integration by parts formula for the integral operator. More precisely, a Rota-Baxter algebra with a preassigned scalar $\lambda$, called the weight, is an associative algebra $A$ with a linear endomorphism $P$ of $A$ satisfying the Rota-Baxter equation:
\begin{equation}
	P(u)P(v)=P(uP(v))+P(P(u)v)+\lambda P(uv) \tforall u, v\in A.
\mlabel{eq:000}
\end{equation}
The analytic model of a Rota-Baxter operator of weight zero is the integral operator 
\begin{equation} \mlabel{eq:int} I(f)(x):=\int_0^x f(t)\,dt, 
	\end{equation}
defined for functions $f$ continuous on $\RR$. Then the integration by parts formula gives 
\begin{equation} \mlabel{eq:ibp}
	\Big(\int_0^xf(t)dt\Big)\Big(\int_0^xg(s)ds\Big) 
	=\int_0^x f(t)\Big(\int_0^t g(s)ds\Big) dt + \int_0^xg(s)\Big(\int_0^s f(t)dt\Big) ds.
\end{equation}
This means that the operator $I$ is a Rota-Baxter operator of weight zero. 
Thus a Rota-Baxter operator of weight zero in general is also called an integral operator. 

While the early developments of Rota-Baxter algebras attracted the attentions of prominent mathematicians such as Rota, Atkinson and Cartier~\mcite{At,Ca,Ro} in the 1960s and 1970s, this century witnesses a remarkable renascence of Rota-Baxter algebras, thanks to their connections to several important areas in mathematics and mathematical physics such as the renormalization of quantum field theory, Yang-Baxter equations, multiple zeta values, combinatorial Hopf algebras and operads~\mcite{Ag,Bai,CK,GK,GLS,Ho,LST,RR,TCGS,YGT}. See~\mcite{Gus,Gub} for a short survey and a more detailed exposition. Furthermore, Rota-Baxter operators have been defined for a wide range of specific algebraic structures and for the general framework of algebraic operads~\mcite{Ag,BBGN,STS}. More recently Rota-Baxter operators have been defined for Hom-Lie algebras, groups, groupoids and cocommutative Hopf algebras~\mcite{Go,GLS,MY}.  

Thus it is natural to define Rota-Baxter operators, in particular integral operators, on lattices and explore their role in the study of lattices. 
This is the purpose of this article. We find it fascinating that properties of lattices that at the outset have nothing to do with differential or \rb operators turn out to be characterized by these operators.  
We also study isomorphic \mrb lattices\footnote{To avoid confusion with the existing notion of integral lattices~\mcite{SJ}, we will use the term \mrb lattices instead. Note however that integral operators only correspond to \mrb operators of weight zero. \mrb lattices from Rota-Baxter operators with nonzero weights will be studied separately.} and classify isomorphism classes of \mrb lattices with certain underlying lattices. We further investigate derived structures from 
differential lattices or \mrb  lattices, motivated by their associative algebra or Lie algebra predecessors which had their origins in hydrodynamics and quantum theory.

By the First Fundamental Theorem of Calculus, the differential operator and the integral operator in analysis are one sided inverses of each other. An analog of this relation also holds for the differential operators and \rb operators~ on associative algebras~\mcite{GK3}. Thus it is also interesting to study the relationship between the corresponding operators on lattices. We found that in place of a formal analogy, the two operators are closely related in another way: an operator $d$
on a lattice is both  an \rb operator and a differential operator in the sense of \mcite{xin1}  if and only if $d$ satisfies Eq.~\meqref{eq:00}, that is, $d$ is a derivation in the sense of Szasz \mcite{sz} (see Proposition \mref{p:00}).

The importance of the differential operator and \rb operator on associative algebras relies on their close relationship with other useful algebraic structures such as Novikov algebras and dendriform algebras. We show that such relations can be extended to the operators on lattices,
see Proposition \mref{pro:887} and Proposition \mref{pro:888}.

These properties show that derivations and integral operators are useful tools to study lattices, as well as to give rise to new structures of independent interests. 

Overall, the paper is organized as follows. In Section~\mref{sec:rb},
 the notions of an \rb operator on a lattice and \mrb lattice are introduced. Each lattice carries several classes of \rb operators, giving a large selection of \mrb lattices. 
We find that an operator on a lattice is both a derivation and an \rb operator~ if and only if it 
satisfies Eq.~\meqref{eq:00} (Proposition \mref{p:00}).
We also characterize some special lattices, such as distributive lattices, weak modular lattices or chains, via \rb operators (Theorem~\mref{pro:000}, Theorem~\mref{the:0002} and Theorem \mref{th:300}).

Section~\mref{sec:inl} studies isomorphism classes of \mrb lattices.
We classify isomorphic \mrb lattices on two types of underlying lattices: the finite chains and the diamond type lattices $M_{n}$ (Proposition \mref{c:3001} and Theorem \mref{te:111}). Their enumerations are related to the Fibonacci numbers. 

As noted above, differential and \rb operators~ on associative algebras gives rise to interesting algebraic structures such as Novikov algebras and dendriform algebras. 
We show in Section~\mref{ss:extra}, that similar structures can be derived from
differential and \rb operators~ on lattices.
 Let $L$ be a distributive lattice, and $d$ be an isotone derivation on $L$.
 Define
 $x\triangleleft y:=d(x)\wedge y $ for all $ x, y\in L$.
  Then $(L,\vee, \triangleleft)$ is a left Novikov semiring
 (Proposition \mref{pro:887}).
For a \mrb distributive lattice $(L,\vee, \wedge, P)$, define
$ x\prec_P y:=x\wedge P(y)$ and $ x\succ_P y:= P(x)\wedge y $ for all $ x, y \in L$.
Then $(L, \vee, \prec_P,\succ_P)$ is a dendriform semiring
 (Proposition \mref{pro:888}).
\smallskip

\noindent
{\bf Notations. }
Throughout this paper, unless otherwise specified,  we let $(L, \vee, \wedge)$ denote a lattice, and 
let $(L, \vee, \wedge, 0, 1)$ denote a bounded lattice with bottom element $0$ and top element $1$. For elements $a, b$ in a poset $(A, \leq)$, we write $a<b$  if $a\leq b$ with $a\neq b$.

\section {Integral operators  on lattices}
\mlabel{sec:rb}
In this section, we introduce \rb operators on lattices, and
characterize some special lattices, such as distributive lattices and chains, in terms of \rb operators.

\begin{definition}
An operator $P: L\rightarrow L$ on a lattice $L$ is called an \name{\rb operator} if $P$ satisfies the following equations: 
\begin{enumerate}
\item $P(x\vee y)=P(x)\vee P(y)$, and
\item  $P(x)\wedge P(y)=P(P(x)\wedge y)\vee P(x\wedge P(y))$ for all $x,y \in L$.
\end{enumerate}
A lattice equipped with an \rb operator is called an \name{\mrb lattice} (of weight zero). 
\mlabel{d:31}
\end{definition}
See the introduction, especially Eqs.~\meqref{eq:int} and \meqref{eq:ibp}, for the motivation for the term integral operator. 

As pointed out by one of the referees, the term integral lattice has been used to mean a discrete additive subgroup of $\RR^n$ such that the inner product of lattice vectors are all integral~\mcite{SJ}. Thus the term \mrb lattice is used here to avoid confusion. However an integral operator is only a \mrb operator of weight zero. So the \mrb lattice considered here is also for weight zero. Other \mrb lattices will be studied separately.

We let $\RBO(L)$ denote the set of all {\rb operators} on a lattice $L$. To conform to the notion of linear operators, we call a map from $L$ to itself an operator even though there is no linear structure on $L$. 

We give some preliminary examples of \rb operators. Many more examples can be found in Propositions~\mref{pp:22}--\mref{pp:0001}. 
\begin{example}
\begin{enumerate}
\item It is clear that the identity map $ \mrep_{L}$ on the lattice $L$ is an \rb operator.
\item Let $L$ be a  lattice. For any  $a\in L$, define an operator $\textbf{C}_{(a)}: L\rightarrow L$ by:
$\textbf{C}_{(a)}(x):=a$ for any $x\in L$. It is easy to see that $\textbf{C}_{(a)}$ is in $\RBO(L)$. 
$\textbf{C}_{(a)}$ is called  the \name{constant \rb operator} with value $a$.
 When $L$ is a  lattice with  bottom element $0$, we write  $\textbf{0}_{L}$ for  $\textbf{C}_{(0)}$. 
\mlabel{it:221}
\item Let $(L, \vee, \wedge, 0, 1)$ be a  lattice. Define an operator $\tau: L\rightarrow L$  by:
$$
    \tau(x):=
    \begin{cases}
      0,  & \textrm{if}~ x=0; \\
      1,  & \textrm{otherwise}.
    \end{cases}
    $$
It is routine to verify that $\tau$ is in $\RBO(L)$.
\mlabel{it:222}
\end{enumerate}
\mlabel{exa:300}
\end{example}

\begin{proposition}
Let  $L$ be a  lattice and $P\in\RBO(L)$. Then the following statements hold:
\begin{enumerate}
\item $P$ is isotone:  $P(x)\leq P(y)$ for any $x, y\in L$ with $x\leq y$..
\mlabel{it:3001}
\item $P(x)=P(x\wedge P(x))$ for any $x\in L$.
\mlabel{it:3002}
\item $P$ is idempotent: $P^{2}=P$.
\mlabel{it:3003}
\item $P(x)=P(x\vee P(x))$ for any $x\in L$.
\mlabel{it:3004}
\item $P(P(x)\wedge y)\leq P(x)\wedge P(y)$ for all $x, y\in L$.
\mlabel{it:3005}
\end{enumerate}
\mlabel{p:300}
\end{proposition}

\begin{proof}
Assume that $L$ is a lattice and $P\in \RBO(L)$. Let $x, y\in L$.

\mnoindent
\mref{it:3001} If $x\leq y$,
then $P(y)=P(x\vee y)=P(x)\vee P(y)$, and so $P(x)\leq P(y)$. Hence $P$ is isotone.

\mnoindent
\mref{it:3002} By  Definition \mref{d:31}, we have $$P(x)=P(x)\wedge P(x)=P(P(x)\wedge x)\vee P(x\wedge P(x))=P(x\wedge P(x)),$$
that is, \mref{it:3002} holds.

 \mnoindent
\mref{it:3003} 
Since $P$ is isotone, we have  
 $P(x)=P(P(x)\wedge x)\leq P(P(x))=P^{2}(x)$ by \mref{it:3002}.
 Also, since $P(x)\wedge P^{2}(x)=P(P(x)\wedge P(x))\vee P(x\wedge P^{2}(x))=P^{2}(x)\vee P(x\wedge P^{2}(x))$,
we get $P^{2}(x)\leq P(x)\wedge P^{2}(x)\leq P(x)$, and so $P^{2}(x)=P(x)$. Hence $P^{2}=P$.

 \mnoindent
\mref{it:3004} Since $P^{2}(x)=P(x)$ by 
\mref{it:3003}, we have $P(x\vee P(x))=P(x)\vee P^{2}(x)=P(x)$.

\mnoindent
\mref{it:3005} follows immediately from Definition \mref{d:31}.
\end{proof}

Let $L$ be a  lattice and $P$  be an operator on $L$.
Denote the set of all fix points of $P$ by
$\Fix_{P}(L)$:
$$\Fix_{P}(L):=\{x\in L~|~ P(x)=x\}\subseteq L.$$
The following simple fact about idempotent operators can be found in~\cite[Lemma 2.4]{GG}.
\begin{lemma} 
Let $L$ be a lattice and $P$ be an operator on $L$. Then $P$ is idempotent if and only if \emph{$\Fix_{P}(L)$} equals to the image $P(L)$ of $P$. 
	\mlabel{l:200}
\end{lemma}

We also have the following easy consequences. 

\begin{corollary}
Let $L$ be a  lattice and  $P\in \RBO(L)$. Then the following statements hold.
\begin{enumerate}
\item \emph{$\Fix_{P}(L)$}$=P(L)$ for the image $P(L)$ of $P$.
\mlabel{it:31}
\item \emph{$\Fix_{P}(L)$}  is a sublattice of $L$.
\mlabel{it:32}
\end{enumerate}
\mlabel{c:300}
\end{corollary}
\begin{proof}
\mnoindent
\mref{it:31}
 follows  by Proposition \mref{p:300} and Lemma \mref{l:200}.

\mnoindent
\mref{it:32} For any $P(x), P(y)\in \Fix_{P}(L)$, we have
$P(x)\vee P(y)=P(x\vee y)\in P(L)=\Fix_{P}(L)$, and
 $$P(x)\wedge P(y)=P(P(x)\wedge y)\vee P(x\wedge P(y))=P((P(x)\wedge y)\vee (x\wedge P(y)))\in P(L)=\Fix_{P}(L).$$
This shows that $\Fix_{P}(L)$ is closed under the operations $\vee$ and $\wedge$. Thus $\Fix_{P}(L)$ is a sublattice of $L$.
\end{proof}
Applying to integral operators, we see that an integral operators on a lattice is almost never injective or surjective. 
\begin{proposition}
Let $L$ be a  lattice and  $P\in$\emph{$\RBO(L)$}. Then the following statements are equivalent.
\begin{enumerate}
\item \emph{$P=\mrep_{L}$}.
\mlabel{it:1001}
\item $P$ is  injective.
\mlabel{it:1002}
\item $P$ is  surjective.
\mlabel{it:1003}
\end{enumerate}
\mlabel{cor:300}
\end{proposition}
\begin{proof}

It is clear that
\mref{it:1001}$\Rightarrow$\mref{it:1002} and \mref{it:1001}$\Rightarrow$\mref{it:1003}.

\mnoindent
\mref{it:1002}$\Rightarrow$\mref{it:1001} Assume that $P$ is an injective \rb operator. For any $x\in L$, we have
$P(P(x))=P(x)$ by Proposition \mref{p:300} \mref{it:3003},  and so $P(x)=x$. Thus $P=\mrep_{L}$.

\mnoindent
\mref{it:1003}$\Rightarrow$\mref{it:1001} Assume that $P$ is a surjective \rb operator. Then
$\Fix_{P}(L)=P(L)=L$ by Corollary \mref{c:300}, and so  $P=\mrep_{L}$.
\end{proof}

The classical derivation and integration in analysis are related by the First Fundamental Theorem of Calculus (FFTC), which implies that the integration is injective. Since Proposition~\mref{cor:300} shows that an \rb operator~ on a lattice is injective only for the identity map, an analogy of the FFT for lattices is not meaningful. However, as we establish below (Proposition~\mref{p:00} and Theorem~\mref{pro:000}), there are other close relations between \rb operators ~and differential operators on a lattice. 

We first give some notions and properties of differential lattices. 
An operator $d$ on a lattice $L$  is called a \name{derivation} or a \name{differential operator}~\mcite{xin1,GG}
if it satisfies the equation
$$
d(x\wedge y)=(d(x)\wedge y)\vee (x\wedge d(y)) \quad \tforall ~ x, y\in L.
$$
Denote the set of all derivations on $L$ by $\DO(L)$. We recall the following results for later applications.
\begin{lemma}
	\name{\mcite{xin1}} Let $L$ be a lattice, $d\in $\emph{$\DO(L)$} and
	$x, y\in L$. Then the following statements hold.
	\begin{enumerate}
		\item $d(x)\leq x$. In particular, $d(0)=0$ if $L$ has bottom element $0$.
		\mlabel{it:2011}
		\item $d(x)\wedge d(y)\leq x\wedge d(y)\leq d(x\wedge y)$.
		\mlabel{it:2012}
		\item If $x\leq d(u)$ for some $u\in L$, then $d(x)=x$.
		\mlabel{it:2013}
		\item If $L$ has top element $1$ and $d(1)=1$, then \emph{$d=\mrep_{L}$}.
		\mlabel{it:2014}
		\item $d$ is  idempotent.
		\mlabel{it:2015}
	\end{enumerate}
	\mlabel{pro:201}
\end{lemma}

Denote the set of all isotone derivations on $L$ by  $\IDO(L)$.
Also recalled~\mcite{sz} that a map $d:L\to L$ is called a \name{meet-translation} if $d(x\wedge y)=x\wedge d(y)$ for all $x, y\in L$.

\begin{lemma}
	Let $L$ be a lattice. If $d\in$
	\emph{$ \IDO(L)$}, then 
	$d(x\wedge y)=d(x)\wedge d(y)$ for all $x, y\in L$.
	\mlabel{l:100}
\end{lemma}
\begin{proof}
	Assume that  $L$ is a lattice,  $d\in \IDO(L)$ and $x, y\in L$.  Since $x\wedge y\leq x$ and  $x\wedge y\leq y$, we have $d(x\wedge y)\leq d(x)\wedge d(y)$, and so $d(x\wedge y)=d(x)\wedge d(y)$
	by Lemma \mref{pro:201} \mref{it:2012}.
\end{proof}

\begin{remark}
	The converse of Lemma~\mref{l:100}  does not hold.
	For example, let $L$ be a lattice with a bottom element $0$. For a given $u\in L\backslash \{0\}$, the constant \rb operator ~$\textbf{C}_{(u)}$ 
		satisfies the condition $\textbf{C}_{(u)}(x\wedge y)=\textbf{C}_{(u)}(x)\wedge \textbf{C}_{(u)}(y)$
	for any $x, y\in L$. But $\textbf{C}_{(u)}$ is not a derivation since
	$\textbf{C}_{(u)}(0)=u\neq 0$.
	\mlabel{re:000}
\end{remark}

Proposition \mref{t:000} and Proposition \mref{the:000} improve ~\cite[Theorem~3.10]{xin2} and \cite[Theorem~3.18]{xin1} by not requiring that $d$ is a derivation in the hypotheses. 

\begin{proposition}
	Let $L$ be a lattice and $d$ be an operator on $L$. Then $d\in$ \emph{$\IDO(L)$} if and only if $d$ is a meet-translation:
	$d(x\wedge y)=x\wedge d(y)$ for all $x, y\in L$.
	\mlabel{t:000}
\end{proposition}
\begin{proof}
	If $d\in \IDO(L)$, then by Lemma \mref{l:100}, $d(x\wedge y)=d(x)\wedge d(y)$ for all $x, y\in L$, and so $d(x\wedge y)=x\wedge d(y)$ by Lemma \mref{pro:201} \mref{it:2012}.
	
	Convesely, if 	$d(x\wedge y)=x\wedge d(y)$ for all $x, y\in L$, then $d(x\wedge y)=y\wedge d(x)$, and so $d(x\wedge y)=(d(x)\wedge y)\vee (x\wedge d(y))$. Hence $d\in \DO(L)$. Also, $d$ is isotone. In fact, if $x\leq y$, then $d(x)=d(x\wedge y)=x\wedge d(y)\leq d(y)$.
	Thus we get  $d\in \IDO(L)$.
\end{proof}

A natural class of derivations, called \textbf{inner derivations}, are defined by taking, for any given $u\in L$, the map
$$d_u(x):=x\wedge u, \quad \tforall x\in L.$$

\begin{proposition}
	Let $L$ be a lattice with top element $1$
	and $d$ be an operator on $L$. Then
	the following statements are equivalent:
	\begin{enumerate}
		\item $d$ is an isotone derivation.
		\mlabel{it:10001}
		\item $d$ is a meet-translation: $d(x\wedge y)=x\wedge d(y)$ for all $x, y\in L$.
		\mlabel{it:10002}
		\item $d(x)= x\wedge d(1)$ for any $x\in L$.
		\mlabel{it:10003}
		\item $d$ is an inner derivation.
		\mlabel{it:10004}
	\end{enumerate}
	Furthermore, these statements are implied by the linearity of a derivation:
	\begin{enumerate}\addtocounter{enumi}{4}
		\item $d$ is a derivation with the {\bf linearity}
		$d(x\vee y)=d(x)\vee d(y)$ for all $x, y\in L$. 
		\mlabel{it:10005}
	\end{enumerate}
	If $L$ is distributive, then all the five statements are equivalent. 
	\mlabel{the:000}
\end{proposition}

\begin{proof}
	\mnoindent
	\mref{it:10001}$\Leftrightarrow$\mref{it:10002} follows by Proposition \mref{t:000}.
	
	\mnoindent
	\mref{it:10002}$\Rightarrow$\mref{it:10003}
	Assume that \mref{it:10002} holds. Then
	$d(x)=d(x\wedge 1)=x\wedge d(1)$ for any $x\in L$, giving \mref{it:10003}.
	
	\mnoindent
	\mref{it:10003}$\Rightarrow$\mref{it:10004} is clear.
	
	\mnoindent
	\mref{it:10004}$\Rightarrow$\mref{it:10001} follows from \cite[Example 3.8]{xin1}.
	
	By~\mcite{fer}, a derivation $d$ with the linearity implies that $d$ is a meet-translation and hence $d$ satisfies all the conditions~\mref{it:10001} -- \mref{it:10004}. 
	
	For the last statement, assume that $L$ is a distributive lattice with top element $1$. Let $d\in \IDO(L)$, and $x, y\in L$. Then by the equivalence of \mref{it:10003} and \mref{it:10004}, we obtain 
	\[ d(x\vee y)=(x\vee y)\wedge d(1)=(x\wedge d(1))\vee (y\wedge d(1))=d(x)\vee d(y). \]
	Hence condition~\mref{it:10004} implies condition~\mref{it:10005} and thus all the five conditions are equivalent. 
\end{proof}

\begin{proposition}
Let $d$ be an operator on a lattice $L$.
Then $d$ is both a differential operator and an \rb operator if and only if $d$ satisfies Eq.~\meqref{eq:00}, that is, $d$ is a derivation in the sense of
Szasz \mcite{sz}.
\mlabel{p:00}
\end{proposition}
\begin{proof}
If $d$ is both a differential operator and an \rb operator, then clearly
$d$ satisfies  Eq.~\meqref{eq:00}.

Conversely, assume that $d$ satisfies Eq.~\meqref{eq:00}.
Then  $d\in \IDO(L)$. By Proposition \mref{t:000}, we get $d(x\wedge y)=x\wedge d(y)=d(x)\wedge y$ for all $x, y\in L$. It follows from Lemma \mref{l:100} and Lemma \mref{pro:201} that
$$d(x)\wedge d(y)=d(x\wedge  y)=d^{2}(x\wedge y)\vee d^{2}(x\wedge y)=d(d(x)\wedge y)\vee d(x\wedge d(y)),$$ and so
$d\in \DO(L)\cap\RBO(L) $.
\end{proof}

Proposition \mref{p:00} tells us that the intersection $\DO(L)\cap \RBO(L)$ of derivations and integral operators on $L$ is contained in $\IDO(L)$, while the next Theorem \mref{pro:000} says that
$\DO(L)\cap \RBO(L)\neq \IDO(L)$ if  $L$ is not  distributive.
This result and its corollaries show that basic properties of a lattice, such as the distributivity, can be characterized by the differential and \rb operators~ on the lattice. 

\begin{theorem}
 	Let $L$ be a  lattice. Then $L$ is  distributive if and only if every inner derivation is an \rb operator.
 	\mlabel{pro:000}
 \end{theorem}
 \begin{proof}
 Assume that $L$ is a distributive lattice. Let $d$ be an inner derivation and $d=d_{u}$,  where $u\in L$. Then for any $x, y\in L$, we have
 $d(x\vee y)=(x\vee y)\wedge u=(x\wedge u)\vee (y\wedge u)=d(x)\vee d(y)$, and so
 $d\in \RBO(L)$ by Proposition \mref{p:00}.	
 	
 Conversely,	assume that $L$ is not distributive. Then there exist $u, v, w\in L$ such that
 	$(u\vee v)\wedge w\neq (u\wedge w)\vee (v\wedge w)$.
 	Consider the inner derivation $d_{w}$, that is,
 	$d_{w}(x)=x\wedge w$ for any $x\in L$.
 	 Since $d_{w}(u\vee v)=(u\vee v)\wedge w\neq (u\wedge w)\vee (v\wedge w)=d_{w}(u)\vee d_{w}(v)$, we have
 	  $d_{w}\not\in \RBO(L)$.
 \end{proof}
 
Proposition \mref{the:000} and Theorem \mref{pro:000} directly give
\begin{corollary}
Let $L$ be a  lattice with top element $1$. Then $L$ is distributive if and only if \emph{$ \IDO(L)\subseteq \RBO(L)$}.
\mlabel{th:000}
\end{corollary}

In what follows, we present several classes of \rb operators ~on a lattice $L$. They are given by step-type operators.

\begin{proposition}
Let $L$ be a  lattice with top element $1$ and $a, b\in L$ with $b\leq a$.
 Define  an operator  $b^{(a)}: L\rightarrow L$  by
$$
    b^{(a)}(x):=
    \begin{cases}
      b,  & \textrm{if}~ x\leq a; \\
      1,  & \textrm{otherwise}.
    \end{cases}
    $$
Then $b^{(a)}$ is in \emph{$\RBO(L)$}.
In particular, if $(L, \vee, \wedge, 0, 1)$ is a bounded lattice, then
$0^{(a)}\in$\emph{$\RBO(L)$} and $0^{(0)}=\tau$ for the map $\tau$ defined in Example~\mref{exa:300} ~\mref{it:222}.
\mlabel{pp:22}
\end{proposition}
\begin{proof}
Let $L$ be a  lattice with top element $1$ and $a, b\in L$ with $b\leq a$.
It is obvious that $b^{(a)}$ is isotone. Also, since $b^{(a)}(L)=\{b, 1\}=\Fix_{b^{(a)}}(L)$,
we have $(b^{(a)})^{2}=b^{(a)}$ by Lemma \mref{l:200}. To prove $b^{(a)}\in \RBO(L)$,
let $x, y\in L$.

If $x\nleqslant a$ or  $y\nleqslant a$, without loss of generality, take $x\nleqslant a$. Then $ b^{(a)}(x)=1$ and $x\vee y\nleq a$. So $b^{(a)}(x\vee y)=1=  b^{(a)}(x)\vee  b^{(a)}(y)$.
 Since $b^{(a)}$ is isotone and $(b^{(a)})^{2}=b^{(a)}$, we have
 $b^{(a)}(x\wedge b^{(a)}(y))\leq b^{(a)}( b^{(a)}(y))=b^{(a)}(y)$, which together with $ b^{(a)}(x)=1$, implies that
  $$ b^{(a)}(x)\wedge b^{(a)}(y)= b^{(a)}(y)=b^{(a)}(y)\vee b^{(a)}(x\wedge b^{(a)}(y))=
   b^{(a)} (b^{(a)}(x)\wedge y)\vee b^{(a)}(x\wedge b^{(a)}(y)).$$

If $x\leq a$ and  $y\leq a$,  then $ b^{(a)}(x)=b^{(a)}(y)=b$ and $x\vee y\leq a$. So  $ b^{(a)}(x\vee y)=b=  b^{(a)}(x)\vee  b^{(a)}(y)$. Also, since $b^{(a)}(x)\wedge y\leq b^{(a)}(x)=b\leq a$ and
 $x\wedge b^{(a)}(y) \leq b^{(a)}(y)=b\leq a$, we have $ b^{(a)} (b^{(a)}(x)\wedge y)= b^{(a)}(x\wedge b^{(a)}(y))=b$. Thus
 $ b^{(a)}(x)\wedge b^{(a)}(y)=b=
   b^{(a)} (b^{(a)}(x)\wedge y)\vee b^{(a)}(x\wedge b^{(a)}(y))$.

Therefore we obtain that  $b^{(a)}$ is in $\RBO(L)$. 
\end{proof}

\begin{proposition}
	Let $L$ be a  lattice with top element $1$ and let $a\in L$.
	Define  an operator  $\tau^{(a)}: L\rightarrow L$  by
	$$
	\tau^{(a)}(x):=
	\begin{cases}
	x,  & \textrm{if}~ x\leq a; \\
	1,  & \textrm{otherwise}.
	\end{cases}
	$$
	Then $\tau^{(a)}$ is in \emph{$\RBO(L)$}.
	\mlabel{pp:20002}
\end{proposition}
\begin{proof}
	Let $L$ be a  lattice with top element $1$ and $a\in L$.
	It is obvious that $\tau^{(a)}$ is isotone. Also, since $\tau^{(a)}(L)=\{x\in L~|~x\leq a\} \cup \{ 1\}=\Fix_{\tau^{(a)}}(L)$,
	we have $(\tau^{(a)})^{2}=\tau^{(a)}$ by Lemma \mref{l:200}. To prove that $\tau^{(a)}$ is in $\RBO(L)$,
	let $x, y\in L$.
	
	If $x\nleqslant a$ or  $y\nleqslant a$, say  $x\nleqslant a$, then $ \tau^{(a)}(x)=1$ and $x\vee y\nleq a$. So  $\tau^{(a)}(x\vee y)=1=  \tau^{(a)}(x)\vee  \tau^{(a)}(y)$.
	Since $\tau^{(a)}$ is isotone and $(\tau^{(a)})^{2}=\tau^{(a)}$, we have
	$\tau^{(a)}(x\wedge \tau^{(a)}(y))\leq \tau^{(a)}( \tau^{(a)}(y))=\tau^{(a)}(y)$, which, together with $ \tau^{(a)}(x)=1$, implies that
	$$ \tau^{(a)}(x)\wedge \tau^{(a)}(y)= \tau^{(a)}(y)=\tau^{(a)}(y)\vee \tau^{(a)}(x\wedge \tau^{(a)}(y))=
	\tau^{(a)} (\tau^{(a)}(x)\wedge y)\vee \tau^{(a)}(x\wedge \tau^{(a)}(y)).$$
	
	If $x\leq a$ and  $y\leq a$,  then $ \tau^{(a)}(x)=x, \tau^{(a)}(y)=y$ and $x\vee y\leq a$. So  
	$$ \tau^{(a)}(x\vee y)=x\vee y=  \tau^{(a)}(x)\vee  \tau^{(a)}(y).$$ 
Also, since $\tau^{(a)}(x)\wedge y=x\wedge \tau^{(a)}(y)=x\wedge y\leq a$, 
	we have
		$$ \tau^{(a)}(x)\wedge \tau^{(a)}(y)=x\wedge y=
	\tau^{(a)} (\tau^{(a)}(x)\wedge y)\vee \tau^{(a)}(x\wedge \tau^{(a)}(y)).$$
In summary, we conclude that  $\tau^{(a)}$ is in $\RBO(L)$. 
\end{proof}

Proposition~\mref{pp:20002} readily gives
\begin{corollary}
Let $(L, \vee, \wedge, 0, 1)$ be a bounded lattice and $a$ be an atom of $L$.
Define  an operator  $P^{(a)}: L\rightarrow L$  by
$$
    P^{(a)}(x):=
    \begin{cases}
     0,  & \textrm{if}~ x=0; \\
      a,  & \textrm{if}~ x= a; \\
      1,  & \textrm{otherwise}.
    \end{cases}
    $$
Then $P^{(a)}$ is in \emph{$\RBO(L)$}.
\mlabel{pp:202}
\end{corollary}

\begin{proposition}\mlabel{pp:22200}
Let $L$ be a  lattice with top element $1$ and $a, b\in L$ with $b< a<1$.
Define  an operator  $\phi^{(a)}_{(b)}: L\rightarrow L$  by
$$ \phi^{(a)}_{(b)}(x):=\left\{ \begin{array}{ll} b, & \text{ if } x\leq b, \\
	 x,
	  & \text{ if } b<x\leq a, \\ 
	 1, & \text{otherwise}. 
	\end{array} \right. 
	$$
Then $\phi^{(a)}_{(b)}$ is in \emph{$\RBO(L)$} if and only if
$L$ satisfies Condition
\begin{eqnarray}
z \text{ and } b \text{ are comparable for any }z\in L \text{ with } z\leq a. 	\mlabel{eqn:04}
\end{eqnarray}
\end{proposition}
\begin{proof}
Let $L$ be a  lattice with top element $1$ and $a, b\in L$ with $b< a<1$.
	
Assume that $\phi^{(a)}_{(b)}\in \RBO(L)$. Then $\phi^{(a)}_{(b)}$ is isotone by Proposition \mref{p:300}. If there exists $z\in L$ such that $z\leq a$, and $z$ and $b$ are incomparable, then
$\phi^{(a)}_{(b)}(z)=1>a=\phi^{(a)}_{(b)}(a)$, contradicting to the fact that $\phi^{(a)}_{(b)}$ is isotone. Thus $L$ satisfies Condition \meqref{eqn:04}.	
	
Conversely, assume that $L$ satisfies Condition \meqref{eqn:04}.
	It is easy to verify that $\phi^{(a)}_{(b)}$ is isotone. Also, since $\phi^{(a)}_{(b)}(L)=\{x\in L~|~b\leq x \leq a\}\cup\{1\}=\Fix_{\phi^{(a)}_{(b)}}(L)$,
	we have $(\phi^{(a)}_{(b)})^{2}=\phi^{(a)}_{(b)}$ by Lemma \mref{l:200}. To verify that $\phi^{(a)}_{(b)}$ is in $\RBO(L)$,
	consider $x, y\in L$.
	
	If $x\nleqslant a$ or  $y\nleqslant a$, say  $x\nleqslant a$, then $\phi^{(a)}_{(b)}(x)=1$ and $x\vee y\nleq a$. Thus $\phi^{(a)}_{(b)}(x\vee y)=1=  \phi^{(a)}_{(b)}(x)\vee  \phi^{(a)}_{(b)}(y)$.
	Since $\phi^{(a)}_{(b)}$ is isotone and $(\phi^{(a)}_{(b)})^{2}=\phi^{(a)}_{(b)}$, we have
	$\phi^{(a)}_{(b)}(x\wedge \phi^{(a)}_{(b)}(y))\leq \phi^{(a)}_{(b)}( \phi^{(a)}_{(b)}(y))=\phi^{(a)}_{(b)}(y)$, which together with $ \phi^{(a)}_{(b)}(x)=1$, implies that
	$$ \phi^{(a)}_{(b)}(x)\wedge \phi^{(a)}_{(b)}(y)= \phi^{(a)}_{(b)}(y)=\phi^{(a)}_{(b)}(y)\vee \phi^{(a)}_{(b)}(x\wedge \phi^{(a)}_{(b)}(y))=
	\phi^{(a)}_{(b)} (\phi^{(a)}_{(b)}(x)\wedge y)\vee \phi^{(a)}_{(b)}(x\wedge \phi^{(a)}_{(b)}(y)).$$

If $x\leq a$ and $y\leq a$, then by Condition~\meqref{eqn:04}, we only need to consider the following four cases:

\begin{enumerate}
	\item Suppose $b<x\leq a$ and  $b<y\leq a$.
	Then $ \phi^{(a)}_{(b)}(x)=x, \phi^{(a)}_{(b)}(y)=y$ and $b<x\vee y\leq a$.
	It follows that  $\phi^{(a)}_{(b)}(x\vee y)=x\vee y=  \phi^{(a)}_{(b)}(x)\vee  \phi^{(a)}_{(b)}(y)$. Also, since $b\leq x\wedge y=\tau^{(a)}_{(b)}(x)\wedge y\leq a$ and
	$b\leq x\wedge y=x\wedge \phi^{(a)}_{(b)}(y) \leq a$, we have $ \phi^{(a)}_{(b)} (\phi^{(a)}_{(b)}(x)\wedge y)=\phi^{(a)}_{(b)}(x)\wedge y=x\wedge y$ and 
	$ \phi^{(a)}_{(b)} (x\wedge\phi^{(a)}_{(b)} (y))=x\wedge\phi^{(a)}_{(b)}(y)=x\wedge y$. Thus
	$$ \phi^{(a)}_{(b)}(x)\wedge \phi^{(a)}_{(b)}(y)= x\wedge y=
\phi^{(a)}_{(b)} (\phi^{(a)}_{(b)}(x)\wedge y)\vee \phi^{(a)}_{(b)}(x\wedge \phi^{(a)}_{(b)}(y)).$$

\item Suppose $b<x\leq a$ and  $y\leq b$.  Then $ \phi^{(a)}_{(b)}(x)=x, \phi^{(a)}_{(b)}(y)=b$ and $b<x\vee y\leq a$.
It follows that  $\phi^{(a)}_{(b)}(x\vee y)=x\vee y=x=x\vee b=  \phi^{(a)}_{(b)}(x)\vee  \phi^{(a)}_{(b)}(y)$. Also, since $\phi^{(a)}_{(b)}(x)\wedge y\leq y \leq b$ and
$x\wedge \phi^{(a)}_{(b)}(y) \leq \phi^{(a)}_{(b)}(y)= b$, we have $ \phi^{(a)}_{(b)} (\phi^{(a)}_{(b)}(x)\wedge y)=b= \phi^{(a)}_{(b)} (x\wedge\phi^{(a)}_{(b)}(y))$. Then
$$ \phi^{(a)}_{(b)}(x)\wedge \phi^{(a)}_{(b)}(y)= x\wedge b=b=
\phi^{(a)}_{(b)} (\phi^{(a)}_{(b)}(x)\wedge y)\vee \phi^{(a)}_{(b)}(x\wedge \phi^{(a)}_{(b)}(y)).$$

\item Suppose $x\leq b$ and  $b<y\leq a$.  Then similarly, we have
$$ \phi^{(a)}_{(b)}(x)\wedge \phi^{(a)}_{(b)}(y)= 
\phi^{(a)}_{(b)} (\phi^{(a)}_{(b)}(x)\wedge y)\vee \phi^{(a)}_{(b)}(x\wedge \phi^{(a)}_{(b)}(y)).$$

\item Suppose $x\leq b$ and  $y\leq b$. Then $\phi^{(a)}_{(b)}(x)= \phi^{(a)}_{(b)}(y)=b$ and $x\vee y\leq b$.
It follows that  $\phi^{(a)}_{(b)}(x\vee y)=b=  \phi^{(a)}_{(b)}(x)\vee  \phi^{(a)}_{(b)}(y)$. Also, since $\phi^{(a)}_{(b)}(x)\wedge y\leq y\leq b$ and
$x\wedge \phi^{(a)}_{(b)}(y) \leq x\leq b$, we have $ \phi^{(a)}_{(b)} (\phi^{(a)}_{(b)}(x)\wedge y)=b=\phi^{(a)}_{(b)} (x\wedge\phi^{(a)}_{(b)} (y))$. Then
$$ \phi^{(a)}_{(b)}(x)\wedge \phi^{(a)}_{(b)}(y)= b=
\phi^{(a)}_{(b)} (\phi^{(a)}_{(b)}(x)\wedge y)\vee \phi^{(a)}_{(b)}(x\wedge \phi^{(a)}_{(b)}(y)).$$
\end{enumerate}
	
In summary, we conclude that  $\phi^{(a)}_{(b)}\in \RBO(L)$. 
\end{proof}

\begin{proposition}
	Let $L$ be a  lattice with top element $1$ and $a, b\in L$ with $b< a$.
	Define  an operator  $\tau^{(a)}_{(b)}: L\rightarrow L$  by:
	$$ \tau^{(a)}_{(b)}(x):=\left\{ \begin{array}{ll} x, & \text{ if } x\leq b, \\
	b,
	& \text{ if } b<x\leq a, \\ 
	1, & \text{otherwise}. 
	\end{array} \right . 
	$$
	Then $\tau^{(a)}_{(b)}$ is in \emph{$\RBO(L)$} if and only if
	$L$ satisfies  Condition \meqref{eqn:04}.	
	\mlabel{pp:22201}
\end{proposition}
\begin{proof}
	Let $L$ be a  lattice with top element $1$ and $a, b\in L$ with $b< a$.
	
	Assume that $\tau^{(a)}_{(b)}\in \RBO(L)$. Then $\tau^{(a)}_{(b)}$ is isotone by Proposition \mref{p:300}. If there exists $z\in L$ such that $z\leq a$, and $z$ and $b$ are incomparable, then
	$\tau^{(a)}_{(b)}(z)=1>b=\tau^{(a)}_{(b)}(a)$, contradicting with the fact that $\tau^{(a)}_{(b)}$ is isotone. Thus $L$ satisfies Condition \meqref{eqn:04}.	
	
Conversely, assume that $L$ satisfies Condition \meqref{eqn:04}.
	It is easy to verify that $\tau^{(a)}_{(b)}$ is isotone. Also, since $\tau^{(a)}_{(b)}(L)=\{x\in L~|~x \leq b\}\cup\{1\}=\Fix_{\tau^{(a)}_{(b)}}(L)$,
	we have $(\tau^{(a)}_{(b)})^{2}=\tau^{(a)}_{(b)}$ by Lemma \mref{l:200}. To prove that $\tau^{(a)}_{(b)}\in \RBO(L)$,
	let $x, y\in L$.
	
	If $x\nleqslant a$ or  $y\nleqslant a$, say  $x\nleqslant a$, then $\tau^{(a)}_{(b)}(x)=1$ and $x\vee y\nleq a$. Thus $\tau^{(a)}_{(b)}(x\vee y)=1=  \tau^{(a)}_{(b)}(x)\vee  \tau^{(a)}_{(b)}(y)$.
	Since $\tau^{(a)}_{(b)}$ is isotone and $(\tau^{(a)}_{(b)})^{2}=\tau^{(a)}_{(b)}$, we have
	$\tau^{(a)}_{(b)}(x\wedge \tau^{(a)}_{(b)}(y))\leq \tau^{(a)}_{(b)}( \tau^{(a)}_{(b)}(y))=\tau^{(a)}_{(b)}(y)$, which together with $ \tau^{(a)}_{(b)}(x)=1$, implies that
	$$ \tau^{(a)}_{(b)}(x)\wedge \tau^{(a)}_{(b)}(y)= \tau^{(a)}_{(b)}(y)=\tau^{(a)}_{(b)}(y)\vee \tau^{(a)}_{(b)}(x\wedge \tau^{(a)}_{(b)}(y))=
	\tau^{(a)}_{(b)} (\tau^{(a)}_{(b)}(x)\wedge y)\vee \tau^{(a)}_{(b)}(x\wedge \tau^{(a)}_{(b)}(y)).$$
	
	If $x\leq a$ and $y\leq a$, then by Condition \meqref{eqn:04}, we only need to consider the following four cases:
	
	\begin{enumerate}
		\item Suppose $b<x\leq a$ and  $b<y\leq a$.
		Then $ \tau^{(a)}_{(b)}(x)=\tau^{(a)}_{(b)}(y)=b$ and $b<x\vee y\leq a$.
		It follows that  $\tau^{(a)}_{(b)}(x\vee y)=b=  \tau^{(a)}_{(b)}(x)\vee  \tau^{(a)}_{(b)}(y)$. Also, since $\tau^{(a)}_{(b)}(x)\wedge y\leq \tau^{(a)}_{(b)}(x)=b$ and
		$x\wedge \tau^{(a)}_{(b)}(y) \leq \tau^{(a)}_{(b)}(y)=b$, we have $ \tau^{(a)}_{(b)} (\tau^{(a)}_{(b)}(x)\wedge y)=\tau^{(a)}_{(b)}(x)\wedge y=b\wedge y=b$ and 
		$ \tau^{(a)}_{(b)} (x\wedge\tau^{(a)}_{(b)} (y))=x\wedge\tau^{(a)}_{(b)}(y)=x\wedge b=b$. Thus
		$$ \tau^{(a)}_{(b)}(x)\wedge \tau^{(a)}_{(b)}(y)= b=
		\tau^{(a)}_{(b)} (\tau^{(a)}_{(b)}(x)\wedge y)\vee \tau^{(a)}_{(b)}(x\wedge \tau^{(a)}_{(b)}(y)).$$
		
		\item Suppose $b<x\leq a$ and  $y\leq b$.  Then $ \tau^{(a)}_{(b)}(x)=b, \tau^{(a)}_{(b)}(y)=y$ and $b<x\vee y\leq a$.
		It follows that  $\tau^{(a)}_{(b)}(x\vee y)=b=  \tau^{(a)}_{(b)}(x)\vee  \tau^{(a)}_{(b)}(y)$. Also, since $\tau^{(a)}_{(b)}(x)\wedge y\leq \tau^{(a)}_{(b)}(x)=b$ and
		$x\wedge \tau^{(a)}_{(b)}(y) \leq \tau^{(a)}_{(b)}(y)=y\leq b$, we have $ \tau^{(a)}_{(b)} (\tau^{(a)}_{(b)}(x)\wedge y)=\tau^{(a)}_{(b)}(x)\wedge y=b\wedge y=y$ and 
		$ \tau^{(a)}_{(b)} (x\wedge\tau^{(a)}_{(b)} (y))=x\wedge\tau^{(a)}_{(b)}(y)=x\wedge y=y$. Therefore,
		$$ \tau^{(a)}_{(b)}(x)\wedge \tau^{(a)}_{(b)}(y)= b\wedge y=y=
		\tau^{(a)}_{(b)} (\tau^{(a)}_{(b)}(x)\wedge y)\vee \tau^{(a)}_{(b)}(x\wedge \tau^{(a)}_{(b)}(y)).$$
		
		\item Suppose $x\leq b$ and  $b<y\leq a$.  Then similarly, we have
		$$ \tau^{(a)}_{(b)}(x)\wedge \tau^{(a)}_{(b)}(y)= 
		\tau^{(a)}_{(b)} (\tau^{(a)}_{(b)}(x)\wedge y)\vee \tau^{(a)}_{(b)}(x\wedge \tau^{(a)}_{(b)}(y)).$$
		
		\item Suppose $x\leq b$ and  $y\leq b$. Then $ \tau^{(a)}_{(b)}(x)=x, \tau^{(a)}_{(b)}(y)=y$ and $x\vee y\leq b$.
		It follows that  $\tau^{(a)}_{(b)}(x\vee y)=x\vee y=  \tau^{(a)}_{(b)}(x)\vee  \tau^{(a)}_{(b)}(y)$. Also, since $\tau^{(a)}_{(b)}(x)\wedge y\leq y\leq b$ and
		$x\wedge \tau^{(a)}_{(b)}(y) \leq x\leq b$, we have $ \tau^{(a)}_{(b)} (\tau^{(a)}_{(b)}(x)\wedge y)=\tau^{(a)}_{(b)}(x)\wedge y=x\wedge y$ and 
		$ \tau^{(a)}_{(b)} (x\wedge\tau^{(a)}_{(b)} (y))=x\wedge\tau^{(a)}_{(b)}(y)=x\wedge y$. Thus
		$$ \tau^{(a)}_{(b)}(x)\wedge \tau^{(a)}_{(b)}(y)= x\wedge y=
		\tau^{(a)}_{(b)} (\tau^{(a)}_{(b)}(x)\wedge y)\vee \tau^{(a)}_{(b)}(x\wedge \tau^{(a)}_{(b)}(y)).$$
	\end{enumerate}
	
Therefore we conclude that  $\tau^{(a)}_{(b)}$ is in $\RBO(L)$. 
\end{proof}

We next turn our attention to a modular lattice $L$, that is, 
$$x\leq y\Rightarrow x\vee (y\wedge z)=y\wedge (x\vee z)\quad \forall x, y, z\in L.$$

\begin{proposition}
Let $L$ be a modular lattice. For $a\in L$, define 
$$\psi_{(a)}:L\to L, \quad \psi_{(a)}(x):=x\vee a, \quad \forall x\in L.$$
Then $\psi_{(a)}$ is in \emph{$\RBO(L)$}. 
\mlabel{pp:0001}
\end{proposition}
\begin{proof}
Assume that $L$ is a modular lattice and $a\in L$.
For any $x, y\in L$, we have $\psi_{(a)}(x)=x\vee a$ and so
$\psi_{(a)}(x\vee y)=x\vee y\vee a=\psi_{(a)}(x)\vee \psi_{(a)}(y)$.
Also, we have   $\psi_{(a)}(x)\wedge \psi_{(a)}(y)=(x\vee a)\wedge (y\vee a)$ and
$$\psi_{(a)}(\psi_{(a)}(x)\wedge y)\vee \psi_{(a)}(x\wedge \psi_{(a)}(y))=[((x\vee a)\wedge y)\vee a]\vee [(x\wedge (y\vee a))\vee a].$$
Since $L$ is a modular lattice, $a\leq x\vee a$ and $a\leq y\vee a$, we have $((x\vee a)\wedge y)\vee a=(x\vee a)\wedge (y\vee a)$ and
$(x\wedge (y\vee a))\vee a=(x\vee a)\wedge (y\vee a)$. Thus
$\psi_{(a)}(x)\wedge \psi_{(a)}(y)=\psi_{(a)}(\psi_{(a)}(x)\wedge y)\vee \psi_{(a)}(x\wedge \psi_{(a)}(y))$, and therefore
 $\psi_{(a)}\in \RBO(L)$.
\end{proof}

 Example \mref{e:2002} shows that the modularity condition in Proposition~\mref{pp:0001} cannot be removed. 

\begin{example}
Let $N_{8}=\{0, a, b, c, u, v, w, 1\}$ be the lattice with its Hasse diagram given by
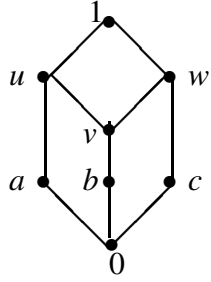
\begin{center}
\setlength{\unitlength}{0.7cm}
\begin{picture}(4,5)
\thicklines
\put(2.0,0.4){\line(-1,1){1.2}}
\put(2.0,0.6){\line(0,1){1.2}}
\put(2.0,0.4){\line(1,1){1.2}}
\put(0.8,1.6){\line(0,2){2.0}}
\put(2.0,1.6){\line(0,1){1.2}}
\put(3.2,1.6){\line(0,1){2.0}}
\put(2.0,2.6){\line(-1,1){1.1}}
\put(2.0,2.6){\line(1,1){1.1}}
\put(1.9,4.7){\line(-1,-1){1.1}}
\put(2.1,4.7){\line(1,-1){1.0}}

\put(2.0,-0.1){$0$}
\put(0.1,1.5){$a$}
\put(1.5,1.5){$b$}
\put(3.5,1.5){$c$}
\put(0.1,3.5){$u$}
\put(1.5,2.4){$v$}
\put(3.5,3.5){$w$}
\put(1.60,4.7){$1$}
\put(1.9,0.3){$\bullet$}
\put(0.6,1.5){$\bullet$}
\put(1.85,1.5){$\bullet$}
\put(3.0,1.5){$\bullet$}
\put(0.6,3.5){$\bullet$}
\put(1.85,2.5){$\bullet$}
\put(3.0,3.5){$\bullet$}
\put(1.85,4.55){$\bullet$}
\end{picture}
\end{center}
It is clear that $L$ is a nonmodular lattice and  $\psi_{(b)}\not\in \RBO(L)$, since
 $\psi_{(b)}(a)\wedge \psi_{(b)}(c)=v$, while
 $\psi_{(b)}(\psi_{(b)}(a)\wedge c)\vee \psi_{(b)}(a\wedge \psi_{(b)}(c))=b$.
 \mlabel{e:2002}
\end{example}

It is remarkable that the counterexample in Example~\mref{e:2002} turns out to be the smallest one for the conclusion of Proposition~\mref{pp:0001}, as we show in the next result. Compare this result with the characterizations of distributive lattices (resp.  modular lattices) by not containing the smallest counterexamples $M_5$ or $N_5$ (resp. $N_5$), see \cite[Theorems~101 and 102]{ge} or \cite[Theorems~3.5 and 3.6]{bu}. It further shows that a weak form of the modularity of a lattice $L$ is charactorized by the \rb operators on $L$

\begin{theorem}
Let $L$ be a  lattice. Then the following statements are equivalent:
\begin{enumerate}
\item 
$(x\vee a)\wedge (y\vee a)=((x\vee a)\wedge y)\vee (x\wedge (y\vee a))\vee a$ for all
$x, y, a\in L$.
\mlabel{it:11}
\item $\psi_{(a)}\in $\emph{$\RBO(L)$} for any $a\in L$.
\mlabel{it:12}
\item $N_{8}$ can not  be embedded into $L$.
\mlabel{it:13}
\end{enumerate}
\mlabel{the:0002}
\end{theorem}
\begin{proof}

\mnoindent
(\mref{it:11} $\Longleftrightarrow$ \mref{it:12}) 
Let $a\in L$. For any $x, y\in L$, we have
 $\psi_{(a)}(x\vee y)=x\vee y\vee a=(x\vee a)\vee (y\vee a)=\psi_{(a)}(x)\vee \psi_{(a)}(y)$.
 So  \begin{eqnarray*}
\psi_{(a)}\in \RBO(L)&\Leftrightarrow& \psi_{(a)}(x)\wedge \psi_{(a)}(y)=\psi_{(a)}(\psi_{(a)}(x)\wedge y)\vee \psi_{(a)}(x\wedge \psi_{(a)}(y))\\
            &\Leftrightarrow& (x\vee a)\wedge (y\vee a)=\Big(((x\vee a)\wedge y)\vee a\Big)\vee \Big((x\wedge (y\vee a))\vee a\Big)\\
             &\Leftrightarrow& (x\vee a)\wedge (y\vee a)=((x\vee a)\wedge y)\vee (x\wedge (y\vee a))\vee a.
\end{eqnarray*}

\mnoindent
(\mref{it:12} $\Longrightarrow$ \mref{it:13}) 
If $N_{8}$ can  be embedded into $L$, then $\psi_{(b)}\not\in \RBO(L)$ according to Example \ref{e:2002}.

\mnoindent
(\mref{it:13} $\Longrightarrow $\mref{it:12}) Assume that  $\psi_{(b)}\not\in \RBO(L)$ for some $b\in L$. Then there exists $a, c\in L$ such that
 $\psi_{(b)}(a)\wedge \psi_{(b)}(c)\neq \psi_{(b)}(\psi_{(b)}(a)\wedge c)\vee \psi_{(b)}(a\wedge \psi_{(b)}(c))$, that is,
 \begin{equation}
	(a\vee b)\wedge (c\vee b)\neq ((a\vee b)\wedge c)\vee (a\wedge (c\vee b))\vee b.
	\mlabel{eq:03}
\end{equation}

Before continuing with the proof, we next establish some preliminary results. 

\textbf{Claim $(i)$}: $b< (a\vee b)\wedge (c\vee b)$. In fact, it is clear that $b\leq (a\vee b)\wedge (c\vee b)$, and
 $$b\leq ((a\vee b)\wedge c)\vee (a\wedge (c\vee b))\vee b\leq (a\vee b)\wedge (c\vee b).$$
 If $b=(a\vee b)\wedge (c\vee b)$, then
 $(a\vee b)\wedge (c\vee b)= ((a\vee b)\wedge c)\vee (a\wedge (c\vee b))\vee b$,
  contradicting Eq. ~\meqref{eq:03}. Thus $b< (a\vee b)\wedge (c\vee b)$.

\textbf{Claim $(ii)$}: $a\vee b$ and $c\vee b$ are  incomparable. In fact, if $a\vee b\leq c\vee b$, then
$(a\vee b)\wedge (c\vee b)=a\vee b$ and $a\leq c\vee b$, which implies that
$$((a\vee b)\wedge c)\vee (a\wedge (c\vee b))\vee b=((a\vee b)\wedge c)\vee (a\vee b)=a\vee b=(a\vee b)\wedge (c\vee b),$$
contradicting Eq.~\meqref{eq:03}. Thus  $a\vee b\not\leq c\vee b$.
Similarly, we can prove that $c\vee b\not\leq a\vee b$.

\textbf{Claim $(iii)$}: $a, b$ and $c$ are mutually incomparable. In fact, if $a\leq b$,
 then $b\vee a=b\leq b\vee c$,
which contradicts {Claim $(ii)$}.
If $b\leq a$, then $(a\wedge c)\vee b\leq a\wedge (c\vee b)$, and so
$$(a\vee b)\wedge (c\vee b)= a\wedge (c\vee b)=
((a\wedge c)\vee b)\vee (a\wedge (c\vee b)) =\Big(((a\vee b)\wedge c)\vee b\Big)\vee (a\wedge (c\vee b)),$$
contradicting Eq.~\meqref{eq:03}.
Thus $b$ and $a$ are incomparable.

If $a\leq c$ or $c\leq a$, then $a\vee b\leq c\vee b$ or $c\vee b\leq a\vee b$,
contradicting {Claim $(ii)$}.

If $c\leq b$, then $(a\vee b)\wedge (c\vee b)= (a\vee b)\wedge b=b$,
contradicting {Claim $(i)$}.

If $b\leq c$, then  $(a\wedge c)\vee b\leq  (a\vee b)\wedge c$. Thus
$$(a\vee b)\wedge (c\vee b)= (a\vee b)\wedge c=
((a\vee b)\wedge c)\vee ((a\wedge c)\vee b)=((a\vee b)\wedge c)\vee (a\wedge (c\vee b))\vee b,$$
contradicting Eq.~\meqref{eq:03}.

Thus $a, b$ and $c$ are mutually incomparable.

\smallskip

With these results established, let $u=b\vee a$, $w=b\vee c$ and $v=u \wedge w=(a\vee b)\wedge (c\vee b)$. Summarizing the above arguments,  we obtain $a\wedge b\wedge c< a< u< u\vee w$,
$a\wedge b\wedge c< b< v=u\wedge w$  and $a\wedge b\wedge c< c< w< u\vee w$.
Also, since $b< v<u$ and $b< v< w$,
we have $u=a\vee b\leq a\vee v\leq a\vee u=u$ and $w=c\vee b\leq c\vee v\leq c\vee w=w$,
which implies that $a\vee v=a\vee b=u$ and $v\vee c=b\vee c=w$.

It is straightforward to verify that the next Hasse diagram gives the desired copy of $N_{8}$ in $L$.
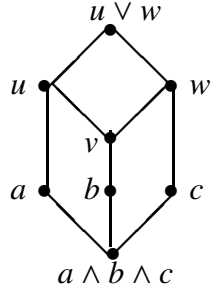
\begin{center}
\setlength{\unitlength}{0.7cm}
\begin{picture}(4.1,5.1)
\thicklines
\put(2.0,0.4){\line(-1,1){1.2}}
\put(2.0,0.6){\line(0,1){1.2}}
\put(2.0,0.4){\line(1,1){1.2}}
\put(0.8,1.6){\line(0,2){2.0}}
\put(2.0,1.6){\line(0,1){1.2}}
\put(3.2,1.6){\line(0,1){2.0}}
\put(2.0,2.6){\line(-1,1){1.1}}
\put(2.0,2.6){\line(1,1){1.1}}
\put(1.9,4.7){\line(-1,-1){1.1}}
\put(2.1,4.7){\line(1,-1){1.0}}

\put(1.0,-0.1){$a\wedge b\wedge c$}
\put(0.1,1.5){$a$}
\put(1.5,1.5){$b$}
\put(3.5,1.5){$c$}
\put(0.1,3.5){$u$}
\put(1.5,2.4){$v$}
\put(3.5,3.5){$w$}
\put(1.60,4.9){$u\vee w$}
\put(1.9,0.3){$\bullet$}
\put(0.6,1.5){$\bullet$}
\put(1.85,1.5){$\bullet$}
\put(3.0,1.5){$\bullet$}
\put(0.6,3.5){$\bullet$}
\put(1.85,2.5){$\bullet$}
\put(3.0,3.5){$\bullet$}
\put(1.85,4.55){$\bullet$}
\end{picture}
\end{center}
\vspace{-.5cm}
\end{proof}

To finish this section, we characterize  chains in terms of \rb operators.
Let $\IEO(L)$ denote the set of all isotone and idempotent operators on a lattice $L$.

\begin{theorem}
Let $L$ be a lattice. Then the following statements are equivalent:
\begin{enumerate}
\item $L$ is a chain.
\mlabel{it:30001}
\item \emph{$\RBO(L)=\IEO(L)$}, that is, the integral operators are precisely the operators that are isotone and idempotent.
\mlabel{it:30002}
\end{enumerate}
\mlabel{th:300}
\end{theorem}
\begin{proof}

By Proposition \mref{p:300} we have  $\RBO(L)\subseteq \IEO(L)$.

\mnoindent
(\mref{it:30001} $\Longrightarrow$ \mref{it:30002})
Assume that $L$ is a chain.
To prove that $ \IEO(L)\subseteq \RBO(L)$, let
 $P\in \IEO(L)$, that is, $P$ is isotone and $P^{2}=P$.
For any $x, y\in L$,
without loss of generality, we can assume that $x\leq y$. Then $P(x)\leq P(y)$ since $P$ is isotone and so
$P(x\vee y)=P(y)=P(x)\vee P(y)$.
	
Also, we have by Proposition \mref{p:300} \mref{it:3002} that
\begin{equation}
P(P(x)\wedge x)=P(x)	
	\mlabel{eq:005}
\end{equation}

Next, we will show that
$P(x)\wedge P(y)=P(P(x)\wedge y)\vee P(x\wedge P(y))$.
In fact,
if $P(x)=P(y)$, then by Eq. \meqref{eq:005}, we have
$$P(x)\wedge P(y)=P(y)\wedge P(x)=P(P(y)\wedge y)\vee P(x\wedge P(x))=P(P(x)\wedge y)\vee P(x\wedge P(y)).$$

If $P(x)\neq P(y)$, then $P(x)< P(y)$ since
$x< y$ and $P$ is isotone, which implies that $P(P(x))=P(x)<P(y)$ and $P(x)<P(y)=P(P(y))$ since $P^{2}=P$.
Noticing that 
$$ P(a)< P(b)\Rightarrow a<b\quad \tforall ~a, b\in L,$$
we get  $P(x)< y$ and $x< P(y)$, so
$P(x)\wedge P(y)=P(x)=P(P(x))\vee P(x)=P(P(x)\wedge y)\vee P(x\wedge P(y))$.

Summarizing the above arguments, we obtain that  $P\in \RBO(L)$. Then
 $\IEO(L)\subseteq \RBO(L)$.  Therefore $\RBO(L)= \IEO(L)$.

\mnoindent
(\mref{it:30002} $\Longrightarrow$ \mref{it:30001}) Assume that $L$ is not a chain. Then there exist $a, b\in L$ such that $a\nleqslant b$
and $b\nleqslant a$. Define an operator $P: L\rightarrow L$ by
 $$
    P(x)=
    \begin{cases}
     a\wedge b,  & \textrm{if}~ x\leq a ~~~~~~ \textrm{or} ~~~~~~ x\leq b; \\
      a\vee b,  & \textrm{otherwise}.
    \end{cases}
    $$
 It is easy to see that $P$ is isotone and $P(L)=\{a\wedge b, a\vee b\}=\Fix_{P}(L)$. Thus $P\in \IEO(L)$ by Lemma \mref{l:200}.
 But $P\not\in \RBO(L)$, since $P(a\vee b)=a\vee b\neq a\wedge b=P(a)\vee P(b)$.
\end{proof}

Let $n$ be a positive integer, and
let $[n]$ denote the set $\{1, 2, \cdots, n\}$ with the standard ordering.
Let $T_{n}$ denote the full transformation semigroup on
 $[n]$, and let $O_{n}$
denote the submonoid of $T_{n}$ consisting of all order-preserving map on $[n]$. Denote by $E(S)$ the set of idempotents of a semigroup $S$, and denote by $|A|$ the cardinality of a set $A$.

\begin{lemma}\cite[Lemma 2.9]{cat}
Let $F_{n}$ be the $n$-th Fibonacci number, defined by the recursion
$F_{1}=F_{2}=1, F_{m}=F_{m-1}+F_{m-2},\, m\geq 3$. Then $|E(O_{n})|=F_{2n}$ for all $n\geq 1$.
\mlabel{lem:224}
\end{lemma}

\begin{corollary}\label{c:3000}
Let 
 $L$ be an $n$-element chain. Then 
 \emph{$| \RBO(L)|=F_{2n}$}, where $F_{2n}$ is the $2n$-th Fibonacci number.
\end{corollary}
\begin{proof}
It follows directly from Theorem \ref{th:300} and Lemma \mref{lem:224}. 
\end{proof}

\section{Isomorphic classes of \mrb lattices} \mlabel{sec:inl}

Recall from Definition~\mref{d:31} that a \mrb lattice (of weight zero) is a lattice equipped with an \rb operator. In this section we study isomorphic \mrb lattices and classify isomorphic \mrb lattices with some common underlying lattices. 

\subsection{Isomorphic \mrb  lattices}
Noting that all axioms of \mrb lattices are equations between terms,
the class of all \mrb lattices forms a variety.
So the notions of isomorphism, subalgebra, congruence and direct product
are directly defined from the corresponding notions in universal algebra~\cite{bu}.

We now study isomorphisms of \mrb lattices before applying it to the classification of some \mrb lattices.

\begin{definition}
Two \mrb lattices  $(L, \vee, \wedge, P)$ and  $(L', \vee', \wedge', P')$ are called \textbf{isomorphic} if there is an isomorphism of lattices $f: L \to L'$ such that $fP=P'f$. When the lattice $L'$ is the same as $L$, we also say that $P$ is \textbf{ isomorphic to} $P'$. We write  $P\cong P'$ if  $P$  is isomorphic to $P'$. 
\end{definition}

It is easy to see that the relation $\cong$ is an equivalence relation on $\RBO(L)$. The corresponding equivalent classes are called \textbf{the isomorphism classes of \rb operators} on $L$. They are the isomorphism classes of \mrb ~ lattices whose underlying lattice is $L$.

Observe that the classification of all isomorphism classes of \mrb lattices is the same as the classification of all isomorphism classes of \mrb lattices or \rb operators on a given underlying lattice, as the underlying lattice runs through isomorphism classes of lattices.	

Lemma \mref{lem:30} tells us that the isomorphism class of the identity operator $\mrep_{L}$ only has one element. 

\begin{lemma}
	Let $L$ be a lattice and $P\in \RBO(L)$.
Then \emph{$P\cong \mrep_{L}$} if and only if \emph{$P= \mrep_{L}$}.	
	\mlabel{lem:30}
\end{lemma}
\begin{proof}
Assume that $P\cong \mrep_{L}$. Then  there exists a lattice automorphism $f: L\rightarrow L$ such that
$Pf=f\mrep_{L}=f=\mrep_{L}f$, which implies that $P= \mrep_{L}$, since $f$ is bijective.	
\end{proof}

Lemma \mref{lem:00} says that the isomorphism classes of $\mathbf{0}_{L}$, $\tau$ and $\mathbf{C}_{(1)}$ only have one element when $L$ is a bounded lattice.

\begin{lemma}
Let $(L, \vee, \wedge,  0, 1)$ be a bounded lattice and \emph{$P, P'\in \RBO(L)$}.
Then the following statements hold.
\begin{enumerate}
\item If  $P\cong P'$, then for $z\in \{0, 1\}$,
$P(1)=z$ if and only if  $P'(1)=z$.
\mlabel{it:001}
\item  \emph{$P\cong \textbf{0}_{L}$} if and only if \emph{$P= \textbf{0}_{L}$}.
\mlabel{it:002}
\item  $P\cong \tau$ if and only if $P= \tau$, where $\tau$ is defined in Example \mref{exa:300} \mref{it:222}.
\mlabel{it:003}
\item  $P\cong \mathbf{C}_{(1)}$ if and only if $P=\mathbf{C}_{(1)} $, where $\mathbf{C}_{(1)}$ is the constant integral operator at value $1$
$($see Example \mref{exa:300} \mref{it:221}$)$.
\mlabel{it:004}
\end{enumerate}
\mlabel{lem:00}
\end{lemma}
\begin{proof}
\mnoindent
\mref{it:001}
Assume that $P, P'\in \RBO(L)$
with  $P\cong P'$.  Then there exists a lattice automorphism $f: L\rightarrow L$ such that $f(P(x))=P'(f(x))$ for any $x\in L$.
If $P(1)=z\in \{0, 1\}$, then
$P'(1)=P'(f(1))=f(P(1))=f(z)=z$, since $f(1)=1$ and $f(0)=0$.
By the symmetry of $P$ and $P'$,  $P'(1)=z\in \{0, 1\}$  implies that  $P(1)=z$.

\mnoindent
\mref{it:002} Assume that $P\cong \textbf{0}_{L}$. Then  there exists a lattice automorphism $f: L\rightarrow L$ such that
$Pf=f\textbf{0}_{L}$. Since $f$ is bijective and $f(0)=0$, we have $f\textbf{0}_{L}=\textbf{0}_{L}=\textbf{0}_{L}f$, and so
$Pf=\textbf{0}_{L}f$. Thus
 $P= \textbf{0}_{L}$, since $f$ is bijective.

\mnoindent
\mref{it:003} Assume that $P\cong \tau$.  Then  there exists a lattice automorphism $f: L\rightarrow L$ such that  $Pf=f\tau$.
Since
$$f(\tau(x)) =
    \begin{cases}
      f(0)=0,  & \textrm{if}~ x= 0; \\
      f(1)=1,  & \textrm{otherwise}
    \end{cases}=\tau(x)$$ and
    $$\tau(f(x)) =
    \begin{cases}
      \tau(0)=0,  & \textrm{if}~ x= 0; \\
      1,  & \textrm{otherwise}
    \end{cases}=\tau(x)$$
 for any $x\in L$, we obtain $f\tau=\tau=\tau f$. Then 
      $Pf= \tau f$. Thus $P=\tau$, since $f$ is  bijective.

\mnoindent
\mref{it:004}   Assume that $P\cong \mathbf{C}_{(1)}$. Then  there exists a lattice automorphism $f: L\rightarrow L$ such that
$Pf=f\mathbf{C}_{(1)}$. Since $f$ is bijective and $f(1)=1$, we have $f\mathbf{C}_{(1)}=\mathbf{C}_{(1)}$, and so
$Pf=\mathbf{C}_{(1)}=\mathbf{C}_{(1)}f$. Thus
 $P= \mathbf{C}_{(1)}$, since $f$ is bijective.
\end{proof}

\subsection{Classification of \rb operators~ on finite chains}

The following lemma says that two \rb operators~ on a finite chain are isomorphic only when they are equal.

\begin{lemma}
Let $L$ be a finite chain. Then an \rb operator~ on $L$ can only be isomorphic to itself. 
\mlabel{p:11}
\end{lemma}
\begin{proof}
 Assume that $L$ is a finite chain and $P, P'\in \RBO(L)$.
It is clear that $P= P'$ implies $P\cong P'$.

Conversely, suppose that   $P\cong P'$. Then  there exists a lattice automorphism $f: L\rightarrow L$ such that
$Pf=fP'$.  Since $f$ is a bijection and both $f$ and $f^{-1}$ are order-preserving
(see Theorem 2.3 in \cite{bu}), we have $f=\mrep_{L}$, and so $P=Pf=fP'=P'$.
\end{proof}

\begin{remark}
	On the other hand, if $L$ is an infinite chain and $P, P'\in \RBO(L)$, then
 $P\cong P'$ does not necessarily imply  $P= P'$.

 For example,  equip the real unit interval $[0, 1]$ with the usual order $\leq$.
 Then $([0, 1], \leq)$ is a chain. Consider the constant operators $\mathbf{C}_{(\frac{1}{2})}$ and $\mathbf{C}_{(\frac{1}{4})}$ on $[0, 1]$, we have
 $\mathbf{C}_{(\frac{1}{2})}, \mathbf{C}_{(\frac{1}{4})}\in \RBO([0, 1])$ by Example \mref{exa:300}, and
 $\mathbf{C}_{(\frac{1}{2})}\neq \mathbf{C}_{(\frac{1}{4})}$, since $\mathbf{C}_{(\frac{1}{2})}(1)=\frac{1}{2}\neq
 \frac{1}{4} =\mathbf{C}_{(\frac{1}{4})}(1)$.
However $\mathbf{C}_{(\frac{1}{2})}\cong \mathbf{C}_{(\frac{1}{4})}$.
 In fact, define an operator $f: [0, 1]\rightarrow [0, 1]$  by $f(x)=x^{2}$ for any $x\in [0, 1]$. Then
 it is easy to see that $f$ is a bijection and both $f$ and $f^{-1}$ are order-preserving. So $f$ is a
 lattice isomorphism by \cite[Theorem 2.3]{bu}. Also, we have
  $f(\mathbf{C}_{(\frac{1}{2})}(x))=f( \frac{1}{2})=( \frac{1}{2})^{2}=\frac{1}{4}=\mathbf{C}_{(\frac{1}{4})}(f(x))$
 for any $x\in [0, 1]$. Thus $\mathbf{C}_{(\frac{1}{2})}\cong \mathbf{C}_{(\frac{1}{4})}$.
\end{remark}

\begin{proposition}
Let 
 $L$ be an $n$-element chain. Then there are exactly $F_{2n}$ isomorphism classes of \rb operators~  on $L$,
  where $F_{2n}$ is the $2n$-th Fibonacci number.
  \mlabel{c:3001}
\end{proposition}
\begin{proof}
It readily follows from Corollary \mref{c:3000} and Lemma \mref{p:11}.
\end{proof}

\subsection{Classification of \rb operators on diamond type lattices}

Let  $M_{n}=\{0, b_{1}, b_{2}, \cdots, b_{n-2}, 1\}$ be the diamond type lattice with Hasse diagram as follows: 

\begin{center}
\setlength{\unitlength}{0.7cm}
\begin{picture}(4,4)
\thicklines
\put(2.0,1.0){\line(-2,1){2.5}}
\put(2.0,1.0){\line(-1,1){1.2}}
\put(2.0,1.0){\line(0,1){1.2}}
\put(2.0,1.0){\line(1,1){1.2}}
\put(2.0,1.0){\line(2,1){2.5}}
\put(2.0,1.0){\line(4,1){4.6}}
\put(2.0,3.5){\line(-2,-1){2.5}}
\put(2.0,3.5){\line(-1,-1){1.2}}
\put(2.0,3.5){\line(0,-1){1.2}}
\put(2.0,3.5){\line(1,-1){1.2}}
\put(2.0,3.5){\line(2,-1){2.5}}
\put(2.0,3.5){\line(4,-1){4.6}}
\put(2.0,0.5){$0$}
\put(-1.2,2.1){$b_{1}$}
\put(0.1,2.1){$b_{2}$}
\put(1.3,2.1){$b_{3}$}
\put(2.5,2.1){$b_{4}$}
\put(3.4,2.1){$\cdots$}
\put(5.0,2.1){$\cdots$}
\put(6.9,2.1){$b_{n-2}$}
\put(2.0,3.7){$1$}
\put(1.9,0.9){$\bullet$}
\put(-0.6,2.1){$\bullet$}
\put(0.6,2.1){$\bullet$}
\put(1.9,2.1){$\bullet$}
\put(3.1,2.1){$\bullet$}
\put(4.3,2.1){$\bullet$}
\put(6.5,2.1){$\bullet$}
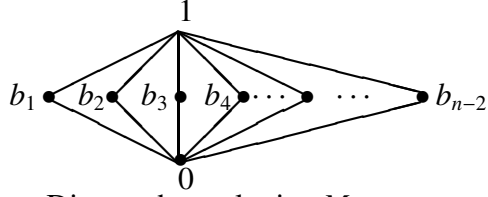
\put(-0.5,0.0){Diamond type lattice $M_{n}$}
\end{picture}
\end{center}
We will determine isomorphism classes of \rb operators~ on $M_{n}$.

\begin{lemma}
Let $n\geq 4$ and $P$ be an operator on the lattice $M_{n}$. If
\emph{$ \Fix_{P}(M_{n})= \{0,  1\}$},
 then \emph{$P\in \RBO(M_{n})$} if and only if $P=0^{(a)}$ for some $a\in M_{n}\backslash \{1\}$, where  $0^{(a)}$ is the \rb operator~ defined in Proposition \mref{pp:22}.
\mlabel{lem:3100}
\end{lemma}
\begin{proof}
Assume that $P$ is an operator on  $M_{n}$, and
$ \Fix_{P}(M_{n})= \{0,  1\}$. Then $P(0)=0$ and $P(1)=1$.

If $P\in \RBO(M_{n})$, then for any $i\in  \{1, 2, \cdots, n-2\}$, we have
$P(b_{i})\in P(L)=\Fix_{P}(M_{n})= \{0,  1\}$ by Corollary \ref{c:300}.
If there exist $k,  \ell\in  \{1, 2, \cdots, n-2\}$ with $k\neq \ell$ such that $P(b_{k})=P(b_{\ell})=0$, then
$1=P(1)=P(b_{k}\vee b_{\ell})=P(b_{k})\vee P(b_{\ell})=0\vee 0= 0$, a contradiction.
Thus there is at most one $k\in  \{1, 2, \cdots, n-2\}$ such that $P(b_{k})=0$, yielding that $P=0^{(a)}$ for some $a\in M_{n}\backslash \{1\}$.

Conversely, if $P=0^{(a)}$ for some $a\in M_{n}\backslash \{1\}$, then
 $P\in \RBO(M_{n})$ by  Proposition \mref{pp:22}.
\end{proof}

\begin{lemma}
Let $n\geq 4$ and $P$ be an operator on the lattice $M_{n}$. If
\emph{$ \Fix_{P}(M_{n})= \{  b_{i}, 1\}$} for some $i\in \{1, 2, \cdots, n-2\}$,
 then \emph{$P\in \RBO(M_{n})$} if and only if $P=\psi_{(b_{i})}$, where $\psi_{(b_{i})}(x)=x\vee b_{i}$ for any $x\in M_{n}$.
 \mlabel{lle:3101}
\end{lemma}
\begin{proof}
Assume that $n\geq 4$, $P$ is an operator on the lattice $M_{n}$ and
$ \Fix_{P}(M_{n})= \{b_{i},  1\}$ for some $i\in \{1, 2, \cdots, n-2\}$. Then $P(b_{i})=b_{i}$ and 
$P(1)=1$.

If $P\in \RBO(M_{n})$, then $P(L)=\Fix_{P}(M_{n})$ by Corollary \mref{c:300}, and so
$P(0), P(b_{j})\in  \{ b_{i}, 1\}$ for any $j\in \{1, 2, \cdots, n-2\}\backslash \{i\}$.
It follows from Proposition \mref{p:300} that $P(b_{j})=P(b_{j}\vee P(b_{j}))=P(1)=1=b_{j}\vee b_{i}$. Also, since $P$ is isotone and $P(0)\in  \{ b_{i}, 1\}$, we obtain that $b_{i}\leq P(0)\leq P(b_{i})=b_{i}$, and so
$P(0)=b_{i}$. 
Thus we have shown that $P(x)=x\vee b_{i}$ for any $x\in M_{n}$, that is,  $P=\psi_{(b_{i})}$.

Conversely, if $P=\psi_{(b_{i})}$, then $P\in \RBO(M_{n})$ by Proposition \mref{pp:0001},  since $M_{n}$ is a modular lattice.
\end{proof}

\begin{lemma}
Let $n\geq 4$ and $P$ be an operator on the lattice $M_{n}$. If
\emph{$ \Fix_{P}(M_{n})= \{ 0, b_{i}, 1\}$} for some $i\in \{1, 2, \cdots, n-2\}$,
 then \emph{$P\in \RBO(M_{n})$} if and only if $P=P^{(b_{i})}$, where $P^{(b_{i})}$ is defined in Corollary \mref{pp:202}, that is,
 $$
    P^{(b_{i})}(x)=
    \begin{cases}
      0,  & \text{if}~ x=0; \\
      b_{i},  & \textrm{if}~ x=b_{i}; \\
      1,  & \textrm{otherwise}.
    \end{cases}
    $$
 \mlabel{lll:001}
\end{lemma}
\begin{proof}
Assume that $n\geq 4$, $P$ is an operator on the lattice $M_{n}$, and
$ \Fix_{P}(M_{n})= \{ 0, b_{i}, 1\}$ for some $i\in \{1, 2, \cdots, n-2\}$. Then $P(0)=0, P(b_{i})=b_{i}$ and 
$P(1)=1$.

If $P\in \RBO(M_{n})$, then $P(L)=\Fix_{P}(M_{n})$ by Corollary \ref{c:300}, and so
$ P(b_{j})\in  \{0, b_{i}, 1\}$ for any $j\in \{1, 2, \cdots, n-2\}\backslash \{i\}$.
Since 
$$1=p(1)=P(b_{i}\vee b_{j})=P(b_{i})\vee P(b_{j})=b_{i}\vee P(b_{j}),$$ 
we have $P(b_{j})=1$. Thus $P=P^{(b_{i})}$.

Conversely, if $P=P^{(b_{i})}$, then since $b_{i}$ is an atom of $M_{n}$, we have $P=P^{(b_{i})}\in \RBO(M_{n})$ by Corollary \mref{pp:202}.
\end{proof}

\begin{lemma}
Let $n\geq 4$ and $k, \ell\in  \{1, 2, \cdots, n-2\}$. Then $0^{(b_{k})}\cong 0^{(b_{\ell})}$, $\psi_{(b_{k})}\cong \psi_{(b_{\ell})}$,
$P^{(b_{k})}\cong P^{(b_{\ell})}$ and $\mathbf{C}_{(b_{k})}\cong \mathbf{C}_{(b_{\ell})}$
in $M_{n}$.
\mlabel{l:3100}
\end{lemma}
\begin{proof}
Assume that $n\geq 4$ and $k, \ell\in  \{1, 2, \cdots, n-2\}$.   Define $f: M_{n}\rightarrow M_{n}$  by
 $$
    f(x)=
    \begin{cases}
      b_{\ell},  & \textrm{if}~ x=b_{k}; \\
      b_{k},  & \textrm{if}~ x=b_{\ell}; \\
      x,  & \textrm{otherwise}.
    \end{cases}
    $$
It is easy to verify that $f$ is a lattice isomorphism.
Since
 $$
    f(0^{(b_{k})}(x))=
    \begin{cases}
     f(0)=0,  & \textrm{if}~ x\leq b_{k}; \\
      f(1)=1,  & \textrm{otherwise}
    \end{cases}=0^{(b_{k})}(x)
    $$
and
$$
    0^{(b_{\ell})}(f(x))=
    \begin{cases}
     0^{(b_{\ell})}(0)=0,  & \textrm{if}~ x=0; \\
      0^{(b_{\ell})}(b_{\ell})=0,  & \textrm{if}~ x=b_{k}; \\
      1,  & \textrm{otherwise}
    \end{cases}\ =0^{(b_{k})}(x)
    $$
we have $f 0^{(b_{k})}=0^{(b_{k})}=0^{(b_{\ell})}f$. Thus
 $0^{(b_{k})}\cong 0^{(b_{\ell})}$.

 Since
 $$
 f(P^{(b_{k})}(x))=
 \begin{cases}
 f(0)=0,  & \textrm{if}~ x=0; \\
  f(b_{k})=b_{\ell},  & \textrm{if}~ x= b_{k}; \\
 f(1)=1,  & \textrm{otherwise}
 \end{cases}
 $$
 and
 $$
 P^{(b_{\ell})}(f(x))=
 \begin{cases}
 P^{(b_{\ell})}(0)=0,  & \textrm{if}~ x=0; \\
 P^{(b_{\ell})}(b_{\ell})=b_{\ell},  & \textrm{if}~ x=b_{k}; \\
 1,  & \textrm{otherwise}
 \end{cases}
 $$
 we have $f P^{(b_{k})}=P^{(b_{\ell})}f$, and thus
 $P^{(b_{k})}\cong P^{(b_{\ell})}$.
 
Also, for any $x\in M_{n}$, we have $$(f\psi_{(b_{k})}) (x)=f(\psi_{(b_{k})} (x))=f(x\vee b_{k})=f(x)\vee f(b_{k})=f(x)\vee b_{\ell}=\psi_{(b_{\ell})}(f(x))=(\psi_{(b_{\ell}}f))(x),$$ and so $f\psi_{(b_{k})}=\psi_{(b_{\ell})}f$. Thus $\psi_{(b_{k})}\cong \psi_{(b_{\ell})}$.

Finally, it is easy to verify that   $f\mathbf{C}_{(b_{k})}= \mathbf{C}_{(b_{\ell})}=\mathbf{C}_{(b_{\ell})} f$.
Hence  $\mathbf{C}_{(b_{k})}\cong \mathbf{C}_{(b_{\ell})}$.
\end{proof}

\begin{lemma}
Let $n\geq 4$ and let $P$ be an operator on the lattice $M_{n}$ for which \emph{$\Fix_P(M_{n})\backslash \{0,1\}$} contains at least two elements. Then $P$ is an \rb operator~ if and only if
\emph{$\{0, 1\}\subseteq \Fix_P(M_{n})$} and $P(b)=1$
for each \emph{$b\in M_{n}\backslash \Fix_P(M_{n})$}. 
		\mlabel{lem:0100}
\end{lemma}
\begin{proof}
	Suppose $\{b_{k}, b_{\ell}\}\subseteq \Fix_{P}(M_{n})$ for some $\{b_{k}, b_{\ell}\}\subseteq M_{n}\backslash \{0, 1\}$ with $b_{k}\neq b_{\ell}$.
	
	If $P\in \RBO(M_{n})$, then $\{0, 1\}\subseteq\Fix_{P}(M_{n})$, since $\Fix_{P}(L)$ is a sublattice of $L$
	by Corollary \mref{c:300}.
	Also, for each
 $b\in M_{n}\backslash \Fix_P(M_{n})$,
	we have $P(b)\in P(L)=\Fix_{P}(M_{n})$ by Corollary \mref{c:300}. If $P(b)=0$, then
	$1=P(1)=P(b\vee b_{k})=P(b)\vee P( b_{k})=P( b_{k})=b_{k}$, a contradiction. Thus
	$P(b)\in \Fix_{P}(M_{n})\backslash \{ 0\}$, and so $b\vee P(b)=1$. It follows from Proposition \mref{p:300} that
	$P(b)=P(b\vee P(b))=P(1)=1$.
	
	Conversely, suppose that $\{0, b_{k}, b_{\ell}, 1\}\subseteq \Fix_{P}(M_{n})$ for some $\{b_{k}, b_{\ell}\}\subseteq  M_{n}\backslash \{0, 1\}$ with $b_{k}\neq b_{\ell}$, and $P(b)=1$
	for each $b\in M_{n}\backslash \Fix_P(M_{n})$.	
	It is easy to see  that $P$ is isotone, and
	$P(x\vee y)=P(x)\vee P(y)$
	for all $x, y\in M_{n}$. Also, since $\Fix_{P}(L)=P(L)$, we have $P^{2}=P$ by Lemma \ref{l:200}.
	
	Next, we show that $P(x)\wedge P(y)=P(P(x)\wedge y)\vee P(x\wedge P(y))$ for all $x, y\in M_{n}$. Noticing that
	$x\leq P(x)$,
	we have $P(x)\wedge P(x)=P(x)=P(P(x)\wedge x)\vee P(x\wedge P(x))$.
	So we may assume that $x\neq y$.

	If $x=1$ or $y=1$, say $x=1$, then $P(x)\wedge P(y)=P(y)=P(y)\vee P^{2}(y)=P(P(x)\wedge y)\vee P(x\wedge P(y))$,
	since $P(1)=1$ and $P^{2}=P$.

	If $x, y\in  \Fix_{P}(M_{n})\backslash \{ 1\}$, then $x\wedge y=0$, $P(x)=x$ and $P(y)=y$. It follows that
	$P(x)\wedge P(y)=x\wedge y=0=0\vee 0=P(P(x)\wedge y)\vee P(x\wedge P(y))$, since $P(0)=0$.
	
	If $x\in  \Fix_{P}(M_{n})\backslash \{ 1\}$ and $y\in M_{n}\backslash \Fix_{P}(M_{n})$, then $x\wedge y=0$, $P(x)=x$ and $P(y)=1$, which implies that
	$P(x)\wedge P(y)=x\wedge 1=x=0\vee x=P(P(x)\wedge y)\vee P(x\wedge P(y))$, since $P(0)=0$.
	
	If $y\in  \Fix_{P}(M_{n})\backslash \{ 1\}$ and $x\in M_{n}\backslash \Fix_{P}(M_{n})$, then 
	we similarly have
	$P(x)\wedge P(y)=P(P(x)\wedge y)\vee P(x\wedge P(y))$.
	
	If $x, y\in  M_{n}\backslash \Fix_{P}(M_{n})$, then $P(x)=P(y)=1$,  and so
	$P(x)\wedge P(y)=1=P(y)\vee P(x)=P(P(x)\wedge y)\vee P(x\wedge P(y))$.
	
	To summarize, we conclude that $P\in \RBO(M_{n})$.
\end{proof}

Immediately from Lemma \mref{lem:0100}, we obtain

\begin{corollary} 
Let $n\geq 4$ and \emph{$P\in \RBO(M_{n})$}. If
\emph{$\{  b_{1}, b_{2}, \cdots, b_{n-2}\}\subseteq \Fix_{P}(M_{n})$}, then
 \emph{ $P=\mrep_{M_{n}}$}.
  \mlabel{cr:00}
\end{corollary}

\begin{lemma}
Let $n\geq 4$ and \emph{$P, P'\in \RBO(M_{n})$}. If \emph{$|\Fix_{P}(M_{n})|=|\Fix_{P'}(M_{n})|\geq 3$},
 then $P\cong P'$.
\mlabel{ll:2222}
\end{lemma}

\begin{proof}
Assume that $n\geq 4$ and $P, P'\in \RBO(M_{n})$.
If $|\Fix_{P}(M_{n})|=|\Fix_{P'}(M_{n})|= 3$, then by Lemma \mref{lem:0100},
$ \Fix_{P}(M_{n})= \{ 0, b_{i}, 1\}$ and $ \Fix_{P'}(M_{n})= \{ 0, b_{\ell}, 1\}$ for some $i, \ell\in \{1, 2, \cdots, n-2\}$.
It follows from Lemma \mref{lll:001} and Lemma \mref{l:3100} that $P\cong P'$.

If $|\Fix_{P}(M_{n})|=|\Fix_{P'}(M_{n})|=k+2\geq 4$, then by Lemma \mref{lem:0100},
we have 
\begin{center}
$ \Fix_{P}(M_{n})= \{ 0,  b_{i_{1}}, b_{i_{2}}, \cdots, b_{i_{k}}, 1\}$ and $ \Fix_{P'}(M_{n})= \{ 0,  b_{j_{1}}, b_{j_{2}}, \cdots, b_{j_{k}}, 1\}$,	
\end{center}
 where $1\leq i_{1}< i_{2}<\cdots <i_{k}\leq n-2$ and $1\leq j_{1}< j_{2}<\cdots <j_{k}\leq n-2$.

Let $f: M_{n}\rightarrow M_{n}$  be a bijection such that $f(0)=0, f(1)=1$ and $f(b_{i_{\ell}})=b_{j_{\ell}}$ for each 
$b_{i_{\ell}}\in \{ b_{i_{1}}, b_{i_{2}}, \cdots, b_{i_{k}}\}$. It is clear that $f$ is an automorphism of $M_{n}$. Also,
by Lemma \mref{lem:0100}, we have
$fP=P'f$.
 Thus $P\cong P'$.
\end{proof}

\begin{lemma}
Let $n\geq 4$,  $P$ an operator on the lattice $M_{n}$ and
\emph{$ \Fix_{P}(M_{n})= \{0,  b_{i}\}$} for some $i\in \{1, 2, \cdots, n-2\}$.
\begin{enumerate}
\item  If $n=4$, then \emph{$P\in \RBO(M_{n})$} if and only if $P=d_{b_{i}}$, where $d_{b_{i}}$ is the inner derivation defined by
$d_{b_{i}}(x)=x\wedge b_{i}$ for any $x\in M_{n}$.
\mlabel{it:31001}
\item  If $n\geq 5$, then \emph{$P\not\in \RBO(M_{n})$}.
\mlabel{it:31002}
\end{enumerate}
\mlabel{lem:3101}
\end{lemma}
\begin{proof}
Suppose that the assumption in the lemma is fulfilled.
Then $P(0)=0$ and $P(b_{i})=b_{i}$.

If $P\in \RBO(M_{n})$, then $P(L)=\Fix_{P}(M_{n})$ by Corollary \ref{c:300}, and so
$P(1), P(b_{j})\in  \{0,  b_{i}\}$ for any $j\in \{1, 2, \cdots, n-2\}\backslash \{i\}$.
Thus $P(1)=b_{i}$ since $P$ is isotone. Also, we have by Proposition \mref{p:300} that
 $P(b_{j})=P(b_{j}\wedge P(b_{j}))=P(0)=0$.
 This shows that $P(x)=x\wedge b_{i}$ for any $x\in M_{n}$, that is, $P=d_{b_{i}}$.

Conversely, if $n=4$ and $P=d_{b_{i}}$, then  $P\in \IDO(L)\subseteq\RBO(L)$ by Proposition \mref{the:000} and Corollary \mref{th:000},
 since $M_{4}$ is a distributive lattice. Thus
 \mref{it:31001} holds.

If $n\geq 5$ and $P=d_{b_{i}}$, then for any $b_{k}, b_{\ell}\in M_{n}\backslash \{0, b_{i}, 1\}$ with $b_{k}\neq b_{\ell}$,
we have $P(b_{k}\vee b_{\ell})=P(1)=b_{i}\neq 0=P(b_{k})\vee P(b_{\ell})$, and so
  $P\not\in \RBO(M_{n})$.
 Thus  \mref{it:31002} holds.
\end{proof}

Here is our classification of isomorphism classes of \rb operators~ on $M_{n}$. 

\begin{theorem}
	\begin{enumerate}
		\item  \emph{$|\RBO(M_{3})|=8$} and there are exactly $8$ isomorphism classes of \rb operators~  on $M_{3}$.
		\mlabel{it:100}	
		\item 
\emph{ $|\RBO(M_{4})|=14$} and there are exactly $9$ isomorphism classes of \rb operators~  on $M_{4}$.	
\mlabel{it:101}	
\item	
Let $n\geq 5$. Then \emph{$|\RBO(M_{n})|=2^{n-2}+3n-4$} and there are exactly $n+4$ isomorphism classes of \rb operators~  on $M_{n}$.
\mlabel{it:102}
\end{enumerate}
\mlabel{te:111}
\end{theorem}

\begin{proof}
It follows from Corollary \mref{c:3000} and Proposition \mref{c:3001} that \mref{it:100} holds.

Let $n\geq 4$ and $P\in \RBO(M_{n})$. Consider the following cases.

\textbf{Case} $(1)$: $|\Fix_{P}(M_{n})|=1$.
In this case, $P$ is equal to one of the following constant operators:
$\mathbf{0}_{M_{n}}$, $\mathbf{C}_{(1)}$ and $\mathbf{C}_{(b_{1})}$, $i\in \{1, 2, \cdots, n-2\}$.
Also, for any $i, j\in  \{1, 2, \cdots, n-2\}$, we have $ \mathbf{C}_{(b_{1})}\cong  \mathbf{C}_{(b_{j})}$ by Lemma \mref{l:3100}.
Thus, in this case, $P$ has $n$ choices, and  by  Lemma \ref{lem:00}, $P$ has $3$ isomorphism classes:
$ \mathbf{0}_{M_{n}}$, $\mathbf{C}_{(1)}, \mathbf{C}_{(b_{1})}$.

\textbf{Case} $(2)$: $|\Fix_{P}(M_{n})|=2$. 
In this case, by Lemmas~\mref{lem:3100}, \mref{lle:3101},
\mref{lem:0100} and \mref{lem:3101}, we have
$ \Fix_{P}(M_{n})= \{0,  1\}$ or  $\{  b_{i}, 1\}$ for some $i\in \{1, 2, \cdots, n-2\}$ if $n \geq 5$; and
$\Fix_{P}(M_{n})= \{0,  1\}$,  $ \{0,  b_{i}\}$ or $\{b_{i}, 1\}$ for some $i\in \{1, 2, \cdots, n-2\}$ if $n=4$.

If $ \Fix_{P}(M_{n})= \{0,  1\}$,
then by Lema \mref{lem:3100},
$P=0^{(a)}$ for some
$a\in M_{n}\backslash \{1\}$. Also, we have by Lemma \mref{l:3100} that $0^{(b_{k})}\cong 0^{(b_{\ell})}$
for any  $k, \ell\in  \{1, 2, \cdots, n-2\}$.

If $ \Fix_{P}(M_{n})= \{  b_{i}, 1\}$ for some $i\in \{1, 2, \cdots, n-2\}$, then
$P=\psi_{(b_{i})}$ by Lemma \mref{lle:3101} , where $\psi_{(b_{i})}(x)=x\vee b_{i}$ for any $x\in M_{n}$.
Also, we have by Lemma \mref{l:3100} that $\psi_{(b_{i})}\cong \psi_{(b_{j})}$ for any  $i, j\in  \{1, 2, \cdots, n-2\}$.

Thus, when $n\geq 5$, $P$ has $(n-1)+(n-2)=2n-3$ choices, and $P$ has $3$ isomorphism classes: $0^{(0)}(=\tau), 0^{(b_{1})}$
and $\psi_{(b_{1})}$, by  Lemma~\mref{lem:00}.

When $n=4$, if $\Fix_{P}(M_{4})=\{0, b_{i}\}$, where $i=1$ or $2$, then by Lemma \mref{lem:3101},
$P$ is equal to $d_{b_{i}}$.
Also, we have $d_{b_{1}}\cong d_{b_{2}}$, since $f=\left( \begin{matrix}
0 & b_{1}  & b_{2} & 1 \\
0 &   b_{2}    & b_{1}     &  1
\end{matrix}
\right)$ is an isomorphism from $(M_{4}, \vee, \wedge, d_{b_{1}}, 0, 1)$ to $(M_{4}, \vee, \wedge, d_{b_{2}}, 0, 1)$.
Thus,  $P$ has $(4-1)+(4-2)+2=7$ choices, and $P$ has $4$ isomorphism classes: $0^{(0)}(=\tau), 0^{(b_{1})}$, $\psi_{(b_{1})}$
and $d_{b_{1}}$  by  Lemma \mref{lem:00}.

\textbf{Case} $(3)$: $|\Fix_{P}(M_{n})|=3$. 

In this case, we have by
Lemma \mref{lem:0100} that
$ \Fix_{P}(M_{n})= \{ 0, b_{i}, 1\}$ for some $i\in \{1, 2, \cdots, n-2\}$,
and so by Lemma \mref{lll:001}, $P=P^{(b_{i})}$.
Thus $P$ has $n-2$ choices, and $P$ has only $1$ isomorphism class by Lemma \ref{l:3100}.

\textbf{Case} $(4)$: $4 \leq|\Fix_{P}(M_{n})|=t\leq n-1$. 

In this case, by
Lemma \ref{lem:0100}, there exist  $1\leq j_{1}<j_{2}<\cdots< j_{k}\leq n-2$ (where $k=t-2$) such that
$\Fix_{P}(M_{n})= \{0,  b_{j_{1}}, b_{j_{2}}, \cdots, b_{j_{k}}, 1\}$
and
$P(b_{i})=1$ for each $b_{i}\in \{b_{1}, b_{2}, \cdots, b_{n-2}\}\backslash\{ b_{j_{1}}, b_{j_{2}}, \cdots, b_{j_{k}}\}$.

Thus in this case, $P$ has $C_{n-2}^{2}+C_{n-2}^{3}+\cdots +C_{n-2}^{n-3}=2^{n-2}-n$ choices, and $P$ has $n-4$ isomorphism classes by
 Lemma \mref{ll:2222}.

\textbf{Case} $(5)$: $|\Fix_{P}(M_{n})|=n$. Then $P=\mrep_{M_{n}}$.

Summarizing the above arguments,
when $n=4$,  we obtain that  there are exactly $4+7+2+1=14$ \rb operators~ on $M_{4}$, and
 there are exactly $3+4+1+1=9$ isomorphism classes of \rb operators~
on $M_{4}$, that is, \mref{it:101} holds.

When $n\geq 5$, we obtain that $|\RBO(M_{n})|=n+(2n-3)+(n-2)+(2^{n-2}-n)+1=2^{n-2}+3n-4$, and there are exactly
$3+3+1+(n-4)+1=n+4$ isomorphism classes of \rb operators~  on $M_{n}$, that is,  \mref{it:102} holds.	
\end{proof}

\section{Derived structures from differential lattices and \mrb lattices}
\mlabel{ss:extra}

The following concepts and results were motivated from studies in hydrodynamics and Lie algebras.

\begin{definition}[Balinsky-Novikov \mcite{BN}, I. Gelfand-Dorfman \mcite{GD1}]
	An algebra $(A,\lnvkv)$, that is, a vector space $A$ with a bilinear binary operation $\lnvkv$, is called  a (left) {\bf Novikov algebra} if
	$$(a,b,c)=(b,a,c),\quad (a\lnvkv b)\lnvkv c=(a\lnvkv c)\lnvkv b\quad \tforall a,b,c\in A.$$
Here $(a,b,c)$ is the associator:
$$ (a,b,c):=(a\lnvkv b)\lnvkv c-a\lnvkv(b\lnvkv c).$$
\end{definition}

\begin{lemma}[S. Gelfand \mcite{GD1}]
	Let $A$ be a commutative associative algebra with a derivation $d$. Define a new operation $\lnvkv $ on $A$ by
	\begin{eqnarray}
		a\lnvkv b:= a d(b)\quad \tforall a, b\in A.
	\end{eqnarray}
	Then $(A,\lnvkv)$ is a left Novikov algebra.
\end{lemma}

We generalize the notion of Novikov algebras to be defined for semirings. 

\begin{definition}
A triple $(L,\vee,\lnvkv)$ is called a (left) {\bf  Novikov semiring} if the binary operation $\vee$ is commutative and associative, the binary operation $\lnvkv$  distributes over $\vee$ and
\begin{equation} ((x\lnvkv y)\lnvkv z) \vee ((y\lnvkv x)\lnvkv z) =
(x\lnvkv (y\lnvkv z))\vee (y\lnvkv (x\lnvkv z)),
\mlabel{eq:lnvkv1}
\end{equation}
\begin{equation}
 (x\lnvkv y)\lnvkv z=(x\lnvkv z)\lnvkv y.
\mlabel{eq:lnvkv2}
\end{equation}
\end{definition}

It is directly checked that every distributive lattice $(L, \vee, \wedge)$ is a Novikov semiring. So there are plenty of examples of Novikov semirings. 

\begin{proposition}
Let $L$ be a distributive lattice,  and
\emph{$d\in \IDO(L)$}.
Define
 	\begin{eqnarray}
x\lnvkv y:=d(x)\wedge y\quad \tforall x, y\in L.
 \end{eqnarray}
 Then $(L,\vee,\lnvkv)$ is a left Novikov semiring.
Moreover,  if $L$ has top element $1$, then
$d$ is a homomorphism of Novikov semirings from $(L,\vee,\wedge)$ to $(L,\vee,\lnvkv)$.
\mlabel{pro:887}
\end{proposition}

\begin{proof}
Assume that $L$ is a distributive,  and $d\in \IDO(L)$. Let $x, y, z\in L$. 
By Proposition \mref{t:000},  we have
$x\lnvkv y=d(x)\wedge y=x\wedge d(y)=y\lnvkv x$, and so 
$\lnvkv$  distributes over $\vee$.

Abbreviate $ xy = x\wedge y$ for now. Since $d(x)y\leq d(x)$, we have by Lemma \mref{pro:201} \mref{it:2013} that
$d(d(x)y)=d(x)y,$
and so
$$ d(d(x)y)z = d(x)yz, \quad d(d(x)z)y=(d(x)z)y=d(x)yz.$$
Hence Eq.~\meqref{eq:lnvkv2} holds.
Further, since $L$ is distributive, we have
$$ ((x\lnvkv y)\lnvkv z) \vee ((y\lnvkv x)\lnvkv z)
= d(d(x)y)z \vee d(d(y)x)z = d(x)yz\vee d(y)xz = d(xy)z$$
and
$$(x\lnvkv (y\lnvkv z))\vee (y\lnvkv (x\lnvkv z))
= d(x)d(y)z.$$
Since $d$ is isotone, it follows from Lemma~\mref{l:100} that Eq.~\meqref{eq:lnvkv1}
holds.
Therefore, $(L,\vee,\lnvkv)$ is a Novikov semiring.

Finally, if $L$ has top element $1$, then for any $x, y\in L$, we have $d(x\vee y)=d(x)\vee d(y)$ by Proposition ~\mref{the:000}, and
$$d(x\wedge y)=d(x)\wedge d(y)=d^{2}(x)\wedge d(y)=d(x)\lnvkv d(y)$$ 
by Lemma~\mref{l:100}.
Consequently, $d$ is a  homomorphism from $(L,\vee,\wedge)$ to $(L,\vee,\lnvkv)$.
\end{proof}

Recall that an associative \name{semiring} is a triple $(A,+, \cdot)$ in which $(A, +)$ is an associative commutative semigroup, $(A, \cdot)$ is an associative semigroup and $\cdot$ is distributive over $+$ from both sides.

\begin{proposition}
Let $L$ be a distributive lattice, and \emph{$P\in \RBO(L)$}. Define
	\begin{eqnarray}
 x\ast_P y:=(x\wedge P(y))\vee (P(x)\wedge y)\quad \tforall x, y\in L.
\end{eqnarray}
Then the following statements hold.
\begin{enumerate}
  \item $P(x\ast_P y)=P(x)\wedge P(y)$ for all $x, y\in L$.
  \mlabel{it:7001}
  \item  $(L, \vee, \ast_P)$ is an associative semiring.
    \mlabel{it:7002}
  \item $P$ is a homomorphism of associative semirings from
$(L,\vee,\ast_{P})$ to $(L,\vee,\wedge)$.
  \mlabel{it:7003}
\end{enumerate}
\end{proposition}
\begin{proof}
Assume that  $L$ is a distributive lattice, and $P\in \RBO(L)$.
 Let $x, y, z\in L$.

\mnoindent
\mref{it:7001} We have $P(x\ast_P y)=P((x\wedge P(y))\vee (P(x)\wedge y))=P(x)\wedge P(y)$ by Definition \mref{d:31}.

\mnoindent
\mref{it:7002}
Since $L$ is distributive, we have by \mref{it:7001} that
\begin{eqnarray*}
(x\ast_{P} y)\ast_{P} z&=& (P(x\ast_{P} y)\wedge z)\vee ((x\ast_{P} y)\wedge P(z))\\
            &=& (P(x)\wedge P(y)\wedge z)\vee \Big(((x\wedge P(y))\vee (P(x)\wedge y))\wedge P(z)\Big)\\
            &=& (P(x)\wedge P(y)\wedge z)\vee (x\wedge P(y)\wedge P(z))\vee (P(x)\wedge y\wedge P(z))\\
             &=& \Big(P(x)\wedge ((P(y)\wedge z)\vee (y\wedge P(z)))\Big)\vee (x\wedge P(y)\wedge P(z))\\
             &=& (P(x)\wedge (y\ast_{P} z))\vee (x\wedge P(y\ast_{P} z))\\
            &=& x\ast_{P} (y\ast_{P} z).    
\end{eqnarray*}
Also, we have $x\ast_{P} y=y\ast_{P}x$ and
\begin{eqnarray*}
x\ast_{P} (y\vee z)&=& (P(x)\wedge (y\vee z))\vee (x\wedge P(y\vee z))\\
            &=&(P(x)\wedge y)\vee (P(x)\wedge z)\vee (x\wedge P(y))\vee (x\wedge P(z)) \\
            &=& \Big((P(x)\wedge y)\vee (x\wedge P(y))\Big)\vee \Big((P(x)\wedge z)\vee (x\wedge P(z))\Big)\\
            &=& (x\ast _{P} y)\vee (x\ast_{P} z).
\end{eqnarray*}
Thus $(L, \vee, \ast_P)$ is an associative semiring.

\mnoindent
\mref{it:7003} Since $P\in \RBO(L)$, we have
$P(x\vee y)=P(x)\vee P(y)$ and $P(x\ast_P y)=P(x)\wedge P(y)$ by \mref{it:7001}. So
$P$ is a homomorphism from
$(L,\vee,\ast_{P})$ to $(L,\vee,\wedge)$. 
\end{proof}

The notion of a dendriform algebra orginated from the work of Loday on algebraic $K$-theory~\mcite{Lo}. It is known that a Rota-Baxter operator (of weight $0$), that is, an integral operator, on an associative algebra gives rise to a dendriform algebra~\mcite{Ag}. As their lattice theoretic analogy, we define 
\begin{definition}
A quadruple $(A,+, \prec, \succ)$ is called a \name{dendriform semiring} if  $(A,+)$ is a semigroup, the binary operations $\prec$ and $ \succ$ are distribute over $+$, and $A$ satisfies the following equations. 
\begin{align}
	&\label{prec}(x\prec y)\prec z=x\prec(y\prec z+y\succ z),\\
	&\label{ps}(x\succ y)\prec z= x\succ (y \prec z),\\
	&\label{succ}x\succ (y\succ z)=(x\prec y+x\succ y)\succ z \quad \text{for all } x, y, z\in A.
\end{align}
\end{definition}

\begin{proposition}
Let $L$ be a distributive lattice, and \emph{$P\in \RBO(L)$}.
 Define
$$ x\prec_P y:=x\wedge P(y), \quad x\succ_P y:= P(x)\wedge y \tforall x, y \in L.
$$
Then $(L, \vee, \prec_P,\succ_P)$ is a dendriform semiring.
\mlabel{pro:888}
\end{proposition}
\begin{proof}
Assume that  $L$ is a distributive lattice, and $P\in \RBO(L)$.
 Let $x, y, z\in L$.

{\bf Claim $(i)$:} $\prec_P$  distributes over $\vee$.
In fact, we have 
$$x\prec_{p}(y\vee z)=x\wedge P(y\vee z)=x\wedge (P(y)\vee P(z))=(x\wedge P(y))\vee (x\wedge P(z))=(x\prec_{P} y)\vee (x\prec _{P}z),$$
$$(y\vee z)\prec_{p}x=(y\vee z)\wedge P(x)=(y\wedge P(x))\vee (z\wedge P(x))=(y\prec_{P} x)\vee (z\prec _{P}x).$$
Therefore $\prec_P$  distributes over $\vee$.

{\bf Claim $(ii)$:} $\succ_P$  distributes over $\vee$.
In fact, we have 
$$x\succ_{p}(y\vee z)=P(x)\wedge (y\vee z)=(P(x)\wedge y)\vee (P(x)\wedge z)=(x\succ_{P} y)\vee (x\succ _{P}z)$$
$$(y\vee z)\succ_{p}x=P(y\vee z)\wedge x=(P(y)\vee P(z))\wedge x=(P(y)\wedge x)\vee (P(z)\wedge x)=(y\succ_{P} x)\vee (z\succ_{P}x).$$
Therefore $\succ_P$  distributes over $\vee$.

{\bf Claim $(iii)$:} $(x\prec_{P} y)\prec_{P} z=x\prec_{P}\Big((y\prec_{P} z)\vee (y\succ_{P} z)\Big)$. In fact, we have
 $$x\prec_{P}\Big((y\prec_{P} z)\vee (y\succ_{P} z)\Big)
=x\wedge P((y\wedge P(z))\vee (P(y)\wedge z))=x\wedge (P(y)\wedge P(z))=(x\prec_{P} y)\prec_{P} z.$$

{\bf Claim $(iv)$:} $(x\succ_{P} y)\prec_{P} z= x\succ_{P} (y \prec_{P} z)$. Indeed, we have
$$(x\succ_{P} y)\prec_{P} z=(P(x)\wedge y)\wedge P(z)=P(x)\wedge (y\wedge P(z))= x\succ_{P} (y \prec_{P} z).$$

{\bf Claim $(v)$:} $x\succ_{P} (y\succ_{P} z)=((x\prec_{P} y)\vee(x\succ_{P} y))\succ_{P} z$. Indeed, we have
\begin{eqnarray*}
x\succ_{P} (y\succ_{P} z)&=&P(x)\wedge (P(y)\wedge z)=(P(x)\wedge P(y))\wedge z\\
&=&P((x\wedge P(y))\vee (P(x)\wedge y))\wedge z=((x\prec_{P} y)\vee(x\succ_{P} y))\succ_{P} z.
\end{eqnarray*}

Summarizing the above calculations, we obtain that $(L, \vee, \prec_P,\succ_P)$ is a dendriform semiring.
\end{proof}

\noindent
{\bf Acknowledgments.}
This work is supported by the NSFC Grants (Nos. 11801239 and 12171022). We thanks the anonymous referees for their helpful suggestions. 

\smallskip

\noindent
{\bf Data Availability Statement.}  Data sharing not applicable to this article as no datasets were generated or analysed during the current study.

\end{document}